\documentclass[a4paper,twoside,10pt]{article}

\usepackage[utf8x]{inputenc}
\usepackage[T1]{fontenc}
\usepackage[english]{babel}
\usepackage{subfigure}
\usepackage{amsmath, amsthm, amsfonts, amssymb}
\usepackage{latexsym}
\usepackage{color}
\usepackage{stmaryrd}
\usepackage{endnotes}
\usepackage{wrapfig}
\usepackage{fancyhdr,lastpage}
\usepackage{ccaption}
\usepackage{xspace}
\usepackage{ifthen}
\usepackage{xkeyval}
\usepackage{remreset}
\usepackage{xcolor}
\usepackage[marginpar]{todo}
\usepackage{calc}
\usepackage{multirow}
\usepackage{graphicx}
\usepackage{bbding}
\usepackage{enumerate}
\usepackage[colorlinks=true, linkcolor=blue, citecolor=magenta]{hyperref}
\usepackage[marginpar]{todo}

\makeatletter

\fancypagestyle{plain}{%
	\fancyhf{} %
	\fancyfoot{}%
}

\newcommand{\resp}{respectively \xspace}
\newcommand{\oas}{$\omega$-almost surely ($\omega$-as)\xspace \renewcommand{\oas}{$\omega$-as\xspace}}
\newcommand{\oeb}{$\omega$-essentially bounded  ($\omega$-eb)\xspace \renewcommand{\oeb}{$\omega$-eb\xspace}}

\theoremstyle{definition}

\newtheorem{defi}{%
	\iflanguage{french}{D\'efinition}{Definition}
}[section]
\newtheorem{remark}{%
	\iflanguage{french}{Remarque}{Remark}%
}[section]

\newtheorem{example}{%
	\iflanguage{french}{Exemple}{Example}%
}[section]

\theoremstyle{plain}

\newtheorem{lemm}{%
	\iflanguage{french}{Lemme}{Lemma}%
}[section]
\newtheorem{prop}{%
	\iflanguage{french}{Proposition}{Proposition}%
}[section]
\newtheorem{theo}{%
	\iflanguage{french}{Th\'eor\`eme}{Theorem}%
}[section]
\newtheorem{coro}{%
	\iflanguage{french}{Corollaire}{Corollary}%
}[section]

\let\c@remark=\c@defi
\let\c@clai=\c@defi
\let\c@lemm=\c@defi
\let\c@prop=\c@defi
\let\c@theo=\c@defi
\let\c@coro=\c@defi
\let\c@pref=\c@defi
\let\c@example=\c@defi

\addto\extrasenglish{%
}

\newcommand{\exs}{\paragraph{Examples:}}

\newcommand{\N}{\mathbf{N}}
\newcommand{\Z}{\mathbf{Z}}

\newcommand{\R}{\mathbf{R}}
\newcommand{\C}{\mathbf{C}}
\renewcommand{\H}{\mathbf{H}}

\newcommand{\intvald}[2]{\left\{ #1,\dots , #2\right\} }

\define@choicekey*{planhyp}{corps}[\val \nr]{r,c,h,}[]{
	\ifcase\nr 
		\def\ph@corps{\R}
	\or
		\def\ph@corps{\C}
	\or
		\def\ph@corps{\H}
	\or
		\def\ph@corps{}
	\fi
}
\presetkeys{planhyp}{corps}{}
\newcommand{\HP}[2][]{
	\begingroup
		\setkeys{planhyp}{#1}
		\ifthenelse{\equal{\ph@corps}{}}
			{\mathbf{H}_{#2}}
			{\mathbf{H}_{#2}\left(\ph@corps \right)}
	\endgroup
}

\newcommand*{\dist}[3][]{
	\ifthenelse{\equal{#1}{}}
		{\left| #2- #3 \right|}
		{
			\ifthenelse{\equal{#1}{SC}}
			{\left\| #2- #3 \right\|}
			{\left| #2- #3 \right|_{#1}}
		}
}
\newcommand*{\distV}[1][]{
	\ifthenelse{\equal{#1}{}}
		{\left| \ . \ \right|}
		{
			\ifthenelse{\equal{#1}{SC}}
			{\left\| \ . \ \right\|}
			{\left| \ . \ \right|_{#1}}
		}
}

\newcommand*{\geo}[3][]{
	\ifthenelse{\equal{#1}{}}
		{\left[ #2, #3 \right]}
		{\left[ #2, #3 \right]_{#1}}
}

\newcommand*{\triang}[4][]{
	\ifthenelse{\equal{#1}{}}
		{\left[ #2, #3, #4 \right]}
		{\left[ #2, #3, #4 \right]_{#1}}
}

\define@boolkey{longueurtrans}{stable}[true]{}
\define@key{longueurtrans}{espace}[]{\def \lt@espace{#1}}
\presetkeys{longueurtrans}{stable=false,espace}{}
\newcommand{\len}[1]{\left\|#1\right\|}

\newcommand{\rips}[3][]{
	\ifthenelse{\equal{#1}{}}
		{P_{#2}\left( #3\right)}
		{P_{#2}^{(#1)}\left( #3\right)}
}
\newcommand{\diam}{\operatorname{diam}}

\newcommand{\sdp}[3][]{
	\ifthenelse{\equal{#1}{}}
		{#2 \rtimes #3}
		{#2 \rtimes _{#1} #3}
}
\newcommand{\aut}[1]{\operatorname{Aut} \left( #1\right)}
\newcommand{\out}[1]{\operatorname{Out} \left( #1\right)}

\newcommand{\free}[1]{\mathbf F_{#1}}

\newcommand{\G}{\mathbf G}
\newcommand{\burn}[2]{\mathbf B _{#1}(#2)}
\newcommand{\dlim}{\displaystyle{\lim_{\longrightarrow}}\ }

\def\restriction#1#2{\mathchoice
              {\setbox1\hbox{${\displaystyle #1}_{\scriptstyle #2}$}
              \restrictionaux{#1}{#2}}
              {\setbox1\hbox{${\textstyle #1}_{\scriptstyle #2}$}
              \restrictionaux{#1}{#2}}
              {\setbox1\hbox{${\scriptstyle #1}_{\scriptscriptstyle #2}$}
              \restrictionaux{#1}{#2}}
              {\setbox1\hbox{${\scriptscriptstyle #1}_{\scriptscriptstyle #2}$}
              \restrictionaux{#1}{#2}}}
\def\restrictionaux#1#2{{#1\,\smash{\vrule height .8\ht1 depth .85\dp1}}_{\,#2}}

\newcommand{\inv}{^{-1}}
\newcommand{\Red}[1]{\operatorname{Red} \left( #1\right)}

\newcommand{\prefF}{\mathcal F}

\renewcommand{\epsilon}{\varepsilon}
\renewcommand{\phi}{\varphi}
\renewcommand{\leq}{\leqslant}
\renewcommand{\geq}{\geqslant}

\newcommand{\makebiblio} {
	\IfFileExists{/Users/coulonr/Maths/svn/papers2/bibliography.bib}{
		\bibliography{/Users/coulonr/Maths/svn/papers2/bibliography}
	}{	
		\bibliography{/Volumes/Data/Maths/svn/papers2/bibliography}	 
	}
	\bibliographystyle{abbrv} 
}

\makeatother

\newcommand{\expo}[2][]{

	\ifthenelse{\equal{#1}{}}

		{\log\left( #2 \right)}
		{\log_{#1}\left( #2 \right)}
}

\begin{document} 

\title{Growth and order of automorphisms of free groups and free Burnside groups}
\author{R\'emi Coulon, Arnaud Hilion}

\maketitle

\begin{abstract}
	We prove that an outer automorphism of the free group is exponentially growing if and only if it induces an outer automorphism of infinite order of free Burnside groups with sufficiently large odd exponent.
\end{abstract}



\tableofcontents

\section{Introduction}

\paragraph{}
Let $n$ be an integer.
A group $G$ has exponent $n$ if for all $g \in G$, $g^n =1$.
In 1902, W.~Burnside asked the following question \cite{Bur02}.
Is a finitely generated group with finite exponent necessarily finite?
In order to study this question, the natural object to look at is the free Burnside group of rank $r$ and exponent $n$.
It is defined to be quotient of the free group $\free r$ of rank $r$ by the (normal) subgroup $\free r^n$ generated by the $n$-th power of all elements.
We denote it by $\burn rn$.
Every finitely generated group with finite exponent is a quotient of a free Burnside group.

\paragraph{}For a long time, hardly anything was known about free Burnside groups.
It was only proved that $\burn rn$ was finite for some small exponents: $n=2$ \cite{Bur02}, $n=3$ \cite{Bur02,LevWae33}, $n=4$ \cite{San40} and $n=6$ \cite{Hal57}.
In 1968, P.S.~Novikov and S.I.~Adian achieved a breakthrough by providing the first examples of infinite Burnside groups \cite{NovAdj68c}.
More precisely, they proved the following theorem.
Assume that $r$ is at least $2$ and $n$ is an odd exponent larger than or equal to 4381 then the free Burnside group of rank $r$ and exponent $n$ is infinite.

\paragraph{}This result has been improved in many directions.
S.I.~Adian decreased the bound on the exponent \cite{Adi79}.
A.Y. Ol'shanski\u\i\ obtained a similar statement using a diagrammatical approach of small cancellation theory \cite{Olc82}.
The case of even exponents has been solved by S.V.~Ivanov \cite{Iva94} and I.G.~Lysenok \cite{Lys96}.
More recently, T.~Delzant and M.~Gromov gave an alternative proof of the infiniteness of Burnside groups \cite{DelGro08}.
To sharpen our understanding of Burnside groups we would like to study the symmetries of $\burn rn$. 
This leads us to the outer automorphism group of $\burn rn$.

\paragraph{}
The subgroup $\free r^n$ is characteristic.
Hence the projection $\free r \twoheadrightarrow \burn rn$ induces a natural homomorphism $\out{\free r} \rightarrow \out{\burn rn}$.
This map is neither one-to-one nor onto. 
However it provides numerous examples of automorphisms of Burnside groups.
For instance the first author proved that for sufficiently large odd exponents, the image of $\out{\free r}$ in $\out{\burn rn}$ contains free subgroups of arbitrary rank and free abelian subgroups of rank $\left\lfloor  r/2 \right\rfloor$ \cite{Cou10a}.
In this article we are interested in the following question.
\paragraph{Question:}
Which (outer) automorphism of $\free r$ induces an (outer) automorphism of infinite order of $\burn rn$?

\begin{remark} 
\label{rem: confusion aut out}
	Since $\burn rn$ is a torsion group, every inner automorphism of $\burn rn$ has finite order.
	Therefore an automorphism $\phi \in \aut{\burn rn}$ has finite order if and only if so has its outer class.
	Hence, for our purpose  we can equivalently  work with $\out{\burn rn}$ or $\aut{\burn rn}$.
\end{remark}

\paragraph{}
The first examples of automorphisms of $\burn rn$ with infinite order were given by E.A.~Cherepanov \cite{Che05}.
In particular he proved that the automorphism $\phi$ of $\mathbf F (a,b)$ given by $\phi(a)=ab$ and $\phi(b)=a$ (also called Fibonacci morphism) induces an automorphism of infinite order of $\burn 2n$ for all odd integers $n \geq 665$.
In \cite{Cou10a}, the first author provides a large class of automorphisms with the same property.
\begin{theo}[Coulon, {\cite[Theorem 1.3]{Cou10a}}]
	Let $\phi$ be a hyperbolic automorphism of $\free r$ (i.e. the semi-direct product $\sdp[\phi]{\free r}\Z$ is word-hyperbolic).
	There exists an integer $n_0$ such that for all odd exponents $n \geq n_0$, $\phi$ induces an element of infinite order of $\out{\burn rn}$.
\end{theo}

\paragraph{}In order to state our main theorem, we first need to recall some definitions about the growth of automorphisms.
Let $\G$ be a finitely generated group endowed with the word-metric.
Given $g \in \G$, the length $\| g \|$ of its conjugacy class is the length of the smallest word over the generators which represents an element conjugated to $g$.
Given an outer automorphism $\Phi$ of $\G$ one says that 
\begin{itemize}
	\item $\Phi$ is \emph{exponentially growing} if there exist $g \in \G$ and $\lambda >1$ such that for all $k \in \N$, $\| \Phi^k(g)\| \geq \lambda^k$,
	\item $\Phi$ is \emph{polynomially growing} if for every $g \in \G$ there is a polynomial $P$ such that for all $k \in \N$, $\| \Phi^k(g)\| \leq P(k)$.
\end{itemize}
This definition actually does not depend on the choice of generators used to compute $\|\Phi^k(g)\|$.
Automorphisms of free groups are either exponentially or polynomially growing.

\paragraph{}The Fibonacci morphism $\phi$ used by E.A.~Cherepanov is not hyperbolic.
It can be seen indeed as the automorphism induced by a pseudo-Anosov homeomorphism of the punctured torus.
Since this homeomorphism preserves the boundary component of the torus, the semi-direct product  $\sdp[\phi] {\free r}\Z$ contains a copy of $\Z^2$ which is an obstruction to being hyperbolic.
Nevertheless like hyperbolic automorphisms it is exponentially growing.
On the other hand we also know that a polynomially growing automorphism of $\free r$ induces an automorphism of finite order of $\burn rn$ for every exponent $n$ \cite{Cou10a}.
It suggests a link between the growth of an automorphism of $\free r$ and its order as automorphism of $\burn rn$.

\begin{theo}\label{theo:main}
Let $\Phi\in\out{\free{r}}$ be an outer automorphism of $\free{r}$.
The following assertions are equivalent:
\begin{enumerate}
	\item 
	\label{enu: mth exp growth}
	$\Phi$ has exponential growth;
	\item 
	\label{enu: mth one exponent}
	there exists $n\in\N$ such that $\Phi$ induces an outer automorphism of
	$\burn{r}{n}$ of infinite order;
	\item 
	\label{enu: mth all odd exponents}
	there exist $\kappa,n_0\in\N$ such that for all odd integers $n\geq n_0$, 
	$\Phi$ induces an outer automorphism of $\burn r{\kappa n}$ of infinite order.
\end{enumerate}
\end{theo}

\paragraph{}
The new result of this article is the implication (\ref{enu: mth exp growth}) $\Rightarrow$ (\ref{enu: mth all odd exponents}).
Before sketching this proof, let us have a look at the arguments used by E.A.~Cherepanov \cite{Che05}.
The proof of the infiniteness of $\burn rn$ by P.S.~Novikov and S.I.~Adian is based on the following important fact \cite{Adi79}.
\begin{prop}
\label{res: criterium Novikov Adian}
	Let $w$ be a reduced word of $\free r$.
	If $w$ does not contain a subword of the form $u^{16}$ then $w$ induces a non-trivial element of $\burn rn$ for all odd exponents $n\geq 655$.
\end{prop}
In particular two distinct reduced words without 8-th power define distinct elements of $\burn rn$.
Compute now the orbit of $b$ under the automorphism $\phi$ of $\mathbf F(a,b)$ defined by $\phi(a)=ab$ and $\phi(b)=a$.
It leads to the following sequence of words 
\begin{displaymath}
	\begin{array}{lclclcl}
	\phi^1(b) & = & a			&\quad 	& \phi^5(b) & = & abaababa\\
	\phi^2(b) & = & ab			&		& \phi^6(b) & = & abaababaabaab\\
	\phi^3(b) & = & aba			&		& \phi^7(b) & = & abaababaabaababaababa\\
	\phi^4(b) & = & abaab		&		& \dots & &
	\end{array}
\end{displaymath}
None of these words  contains a 4-th power \cite{Karhumaki:1983fy}.
Therefore they induce pairwise distinct elements of $\burn rn$.
In particular, as an automorphism of $\burn rn$, $\phi$ has infinite order.

\paragraph{}
This argument can be generalized for any exponentially growing automorphism of $\free 2$ using an appropriate train track representative.
However, it does not work anymore in higher rank.
Consider for instance the exponentially growing automorphism $\psi$ of $\mathbf F (a,b,c,d)$ defined by $\psi(a)=a$, $\psi(b)=ba$, $\psi(c)=cbcd$ and $\psi(d)=c$.
As previously we compute the orbit of $d$ under $\psi$.
\begin{displaymath}
	\begin{array}{lcl}
	\psi^1(d) & = & c			\\
	\psi^2(d) & = & c\mathbf bcd			\\
	\psi^3(d) & = & cbcd\mathbf b\mathbf acbcdc			\\
	\psi^4(d) & = & cbcdbacbcdc\mathbf b\mathbf a^2cbcdbacbcdccbcd	\\
	\psi^5(d) & = & cbcdbacbcdcba^2cbcdbacbcdc^2bcd\mathbf b\mathbf a^3cbcdbacbcdcba^2 \dots \\
	&& \dots cbcdbacbcdc^2bcdcbcdbacbcdc
	\end{array}
\end{displaymath}
This orbit is exponentially growing.
Note that each time $\psi^p(d)$ contains a subword $ba^m$ then $\psi^{p+1}(d)$ contains $ba^{m+1}$.
Hence the $\psi^p(d)$'s contain arbitrarily large powers of $a$.
This cannot be avoided by choosing the orbit of another element.
\autoref{res: criterium Novikov Adian} is no more sufficient to tell us wether or not the $\psi^p(d)$'s are pairwise distinct in $\burn rn$.
Therefore, we need a more accurate criterion to distinguish two different elements of $\burn rn$.
This is done using elementary moves.

\paragraph{}
Let $n \in \N$ and $\xi \in \R_+$.
An \emph{$(n,\xi)$-elementary move} consists in replacing a reduced word of the form $pu^ms \in \free r$ by the reduced representative of $pu^{m-n}s$, provided $m$ is an integer larger than $n/2-\xi$.
Note that an elementary move may increase the length of the word.
\begin{theo}[Coulon {\cite[Theorem 4.12]{Coulon:2012tj}}]
\label{res: elementary moves}
	There exist integers $n_1$ and $\xi$ such that for all odd exponents $n \geq n_1$ we have the following property.
	Let $w$ and $w'$ be two reduced words of $\free r$.
	If $w$ and $w'$ define the same element of $\burn rn$ then there are two sequences of $(n,\xi)$-elementary moves which respectively send $w$ and $w'$ to the same word.
\end{theo}

\paragraph{} Thanks to this tool we can now explain using the example $\psi$ how our proof works.
We need to understand the effect of elementary moves on a word $\psi^p(d)$.
To that end we assign colors to the letters.
Let say that  $a$, $b$ are \emph{yellow} letters (dotted lines on \autoref{fig:yellow-red-intro}) whereas $c$, $d$ are \emph{red} letters (thick lines on the figure).
The word $\psi^p(d)$ is then the concatenation of maximal yellow and red subwords.
We define $\Red {\psi^p(d)}$ to be the red word obtained from $\psi^p(d)$ be removing all the yellow letters.
\begin{figure}[h]
	\centering
	\includegraphics[width=\textwidth]{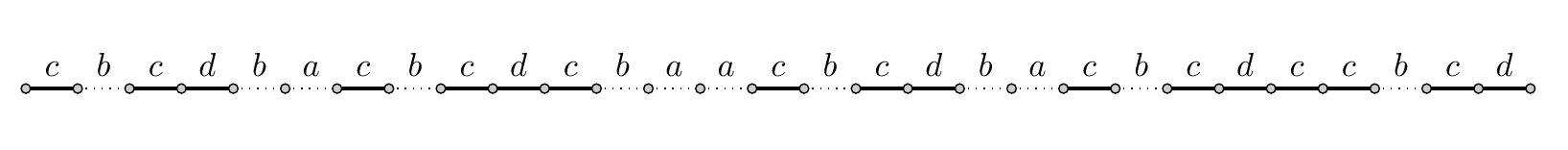}
	\caption{The yellow-red decomposition of $\psi^4(d)$}
	\label{fig:yellow-red-intro}
\end{figure}%
We claim that the elementary moves preserve $\Red {\psi^p(d)}$.
Since the orbit of $d$ grows exponentially one can prove that $\Red{\psi^p(d)}$ does not contain large powers.
More precisely there is an integer $n_2$ such that for all $p \in \N$, $\Red{\psi^p(d)}$ does not contain any $n_2$-th power (see \autoref{res: no power in the red}).
This fact can be interpreted in terms of dynamical properties of the attractive laminations associated to the automorphism $\psi$.
It follows from this remark that the power $u^m$ involved in any $(n,\xi)$-elementary move with $n > 2n_2+2\xi$ only contains yellow letters.
However such a move could completely collapse a yellow subword and then affect the red letters (see \autoref{fig:collapsing-intro})
\begin{figure}[h]
	\centering
	\includegraphics[scale=0.8]{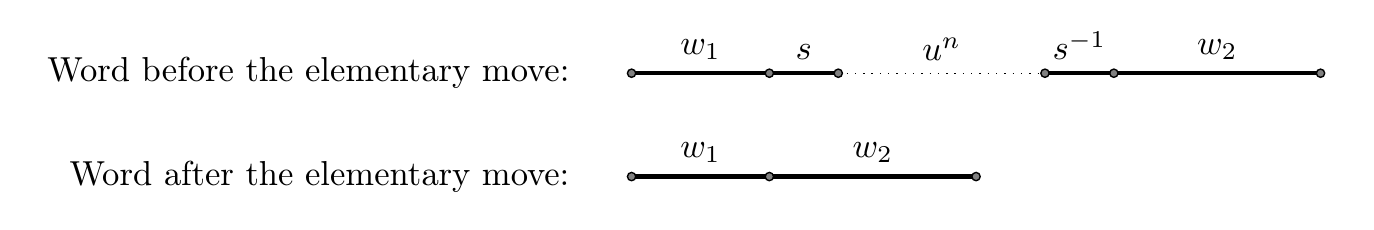}
	\caption{An elementary move collapsing red letters}
	\label{fig:collapsing-intro}
\end{figure}

\paragraph{}
Let us now have a look at the yellow subwords.
Notice that the image by $\psi$ of a yellow word is still a yellow word.
On the contrary, the image of a red word may contain yellow subwords (e.g. $\psi(c)$).
Therefore the yellow subwords of $\psi^p(d)$ can be sorted in two categories.
The ones which are yellow subwords of $\psi(c)$ and $\psi(d)$ and the ones which arise as the images by $\psi$ of yellow subwords of $\psi^{p-1}(d)$.
In particular, among all the $\psi^p(d)$'s, there is only a finite number of orbits under $\psi$ of maximal yellow subwords.
Consequently there is an integer $n_3$ such that for every odd  integer $n >n_3$, none of them becomes trivial in $\burn rn$. 

\paragraph{}We can now argue by contradiction.
Let $n > \max\{n_1, 2n_2+2\xi,n_3\}$ be an odd integer.
Assume that $\psi$ induces an automorphism of finite order of $\burn rn$.
In particular there exists $p \in \N^*$ such that $\psi^p(d)$ and $d$ have the same image in $\burn rn$.
It follows from \autoref{res: elementary moves} that a sequence of $(n,\xi)$-elementary moves sends $\psi^p(d)$ to $d$.
However $n > 2n_2+2\xi$, thus these moves will only change the yellow subwords of  $\psi^p(d)$.
Moreover, since $n > n_3$ none of the yellow words can completely disappear. 
Therefore the red word $\Red{\psi^p(d)}$ associated to $\psi^p(d)$ should be exactly $d$.
Contradiction.

\paragraph{}The proof for an arbitrary exponentially growing automorphism of $\free r$ follows the same ideas.
One has to replace the words in $a,b,c,d$ by paths in an appropriate relative train track.
This leads to a technical difficulty, though.
The red and yellow paths that we want to consider do not necessarily represent elements of the free groups.
This problem is handled in \autoref{sec: finite index subgroup} and \autoref{sec:no big powers}.
We use there subtle aspects of the machinery of train-tracks to show that the red words do not contain large powers (\autoref{res: no power in the red}).
In particular we need to pass to a finite index subgroup of $\free r$. 
This operation actually ensures at the same time that no yellow subpath will be removed by elementary moves (see the discussion before).
Beside this fact, the main ingredients are the ones described above.

\paragraph{Acknowledgment}
Most of this work was done while the first author was staying at the \emph{Max-Planck-Institut f\"ur Mathematik}, Bonn, Germany.
He wishes to express his gratitude to all faculty and staff from the MPIM for their support and warm hospitality. 
The second author would like to thank Michael Handel and Gilbert Levitt for helpful conversations.
The second author is supported by the grant ANR-10-JCJC 01010 of the Agence Nationale de la Recherche

\section{Primitive matrices and substitutions}
\label{sec:primitive}
\paragraph{}In this section we summarize a few properties about primitive integer matrices and substitutions on an alphabet that we be useful later.

\subsection{Primitive matrices}
\label{sec:primitive matrices}

\paragraph{}
A square matrix $M$ of size $\ell$ whose entries are non-negative integers is \emph{irreducible} if for each $i,j\in\{1,\dots,\ell\}$, there exists $p\in\N$ such that the $(i,j)$-entry of $M^p$ is not zero. 
It is \emph{primitive} when there exists $p\in\N$ such that any entry of $M^p$ is not zero.

\paragraph{}
The Perron-Frobenius theorem for an irreducible matrix $M$ with non-negative integer entries states that there exists a unique dominant eigenvalue $\lambda\geq1$ of $M$ associated to an eigenvector with positive coordinates (see for instance Seneta's book \cite{Seneta:2006wd}).
This $\lambda$ is called the \emph{Perron-Frobenius-eigenvalue} (or simply \emph{PF-eigenvalue}) of $M$.
In addition, if $\lambda = 1$ then $M$ is a transitive permutation matrix.

\subsection{Primitive substitutions}
\label{sec:primitive substitutions}

\paragraph{} Let $\mathcal A = \{a_1, \dots, a_\ell\}$ be a finite alphabet.
The free monoid generated by $\mathcal{A}$ is denoted by $\mathcal{A}^*$.
We write $1$ for the empty word, also called \emph{trivial word}.
An infinite word is an element of $\mathcal A^\N$.
Let $m \in \N^*$.
A word $w\in\mathcal{A}^*$ is an \emph{$m$-th power} if there exists a non-trivial word $u\in\mathcal{A}^*$ such that $w=u^m$. 
A non-trivial word $w\in\mathcal{A}^*$ is \emph{primitive} if it is not an $m$-th power with $m$ at least $2$ (i.e. if $w=u^m$, then $u=w$ and $m=1$).
A word $w\in\mathcal{A}^*$ (or an infinite word $w \in \mathcal A^\N$) \emph{contains an $m$-th power}, if there exists a word $u \in \mathcal A^*\setminus\{1\}$ such that $u^m$ is a subword of $w$.
The \emph{shift} is the map $S : \mathcal A^\N \rightarrow \mathcal A^\N$ which sends $(w_i)_{i \in \N}$ to $(w_{i+1})_{i \in \N}$.
An infinite word $w$ is said to be \emph{shift-periodic} if there exists $q\in \N^*$ such that $S^q(w) = w$.
If $u$ stands for the word $w_0w_1\cdots w_{q-1}$ then we write $w=u^\infty$.
Roughly speaking it means that $w$ is the infinite power of $u$.

\paragraph{}A morphism of the free monoid $\mathcal{A}^*$ is called a \emph{substitution} defined on $\mathcal{A}$.
Such a substitution $\sigma$ is indeed completely determined by the images $\sigma(a)\in\mathcal{A}^*$ of all the letters $a\in\mathcal{A}$.
Moreover, it naturally extends to a map $\mathcal A^\N \rightarrow \mathcal A^\N$.
The \emph{matrix $M$ of a substitution} $\sigma$ is a square matrix of size $\ell$ whose $(i,j)$-entry is the number of occurrences of the letter $a_i$ in the word $\sigma(a_j)$.
The substitution $\sigma$ is said to be \emph{primitive} when $M$ is primitive.

\begin{prop}
	\label{res: no big power in the orbit of the substitution}
	Let $a$ be a letter of $\mathcal A$.
	Let $\sigma$ be a primitive substitution on $\mathcal A$ such that $a$ is a prefix of $\sigma(a)$.
	\begin{enumerate}[(i)]
		\item \label{enu: no big power in the orbit of the substitution - converge}
		The sequence $(\sigma^p(a))$ converges to an infinite word $\sigma^\infty(a)$ fixed by $\sigma$.
		\item \label{enu: no big power in the orbit of the substitution - not shift-preiodic}
		If $\sigma^\infty(a)$ is not shift-periodic, then there exists an integer $m \geq 2$ such that for all $p \in \N$, $\sigma^p(a)$ does not contain an $m$-th power.	
		\item \label{enu: no big power in the orbit of the substitution - shift-periodic}
		If there exists a non-trivial primitive word $u$ such that $\sigma^\infty(a)=u^\infty$ then there exists an integer $q\geq2$ such that $\sigma(u)=u^q$.
	\end{enumerate}
\end{prop}

\begin{remark}
\label{rem: primitive shift periodic substitution}
The case covered by Point~(\ref{enu: no big power in the orbit of the substitution - shift-periodic}) in the proposition is not vacuous. 
Consider for instance the substitution defined on $\mathcal A = \{a,b,c\}$ by $\sigma(a) = ab$, $\sigma(b) = c$ and $\sigma(c) = abc$.
The transition matrix $M$ of $\sigma$ and its square are the followings
\begin{equation*}
	M = \left(
	\begin{array}{ccc}
		1 & 0 & 1 \\
		1 & 0 & 1 \\
		0 & 1 & 1 \\
	\end{array}
	\right)
	\quad
	M^2 = \left(
	\begin{array}{ccc}
		1 & 1 & 2 \\
		1 & 1 & 2 \\
		1 & 1 & 2 \\
	\end{array}
	\right)
\end{equation*}
In particular, $\sigma$ is primitive.
However $(\sigma^n(a))$ converges to the infinite shift-periodic word $(abc)^{\infty}$.
\end{remark}

\paragraph{}To prove \autoref{res: no big power in the orbit of the substitution} we use the following results due to B.~Moss\'e.

\begin{prop}[Moss\'e {\cite[Th\'eor\`eme 2.4]{Mos92}}]
	\label{res: proposition mosse}
	Let $\sigma$ be a primitive substitution on a finite alphabet $\mathcal{A}$.
	Let $u\in\mathcal{A}^\N$ be an infinite word fixed by $\sigma$.
	Then either
	\begin{enumerate}[(i)]
	  \item $u$ is shift-periodic, or
	  \item there exists an integer $m \geq 2$ such that $u$ does not contain an $m$-th power. \qedhere
	\end{enumerate}
\end{prop}

\begin{lemm}[Moss\'e {\cite[Proposition 2.3]{Mos92}}]
\label{lem: mosse}
	Let $u\in\mathcal{A}^*$ be a primitive word.
	Let $m \geq 2$ be an integer.
	If $uwu$ is subword of $u^m$, then there exists an integer $p\geq0$ such that $w=u^p$.
	\qedhere
\end{lemm}

\begin{proof}[Proof of \autoref{res: no big power in the orbit of the substitution}]
	By assumption $a$ is a prefix of $\sigma(a)$.
	Thus there exists $w \in \mathcal A^*$ such that $\sigma(a) = aw$.
	Since $\sigma$ is primitive, $w$ is not trivial.
	It follows that for every $p \in \N^*$, $\sigma^p(a)$ is exactly the word
	\begin{displaymath}
		\sigma^p(a) = aw\sigma(w)\sigma^2(w)\dots\sigma^{p-1}(w).	
	\end{displaymath}
	In particular, $\sigma^p(a)$ is a prefix of $\sigma^{p+1}(a)$.
	Therefore, $(\sigma^p(a))$ converges to an infinite word $\sigma^\infty(a)$ fixed by $\sigma$:
	\begin{displaymath}
		\sigma^\infty(a) = aw\sigma(w)\sigma^2(w)\dots\sigma^p(w)\dots
	\end{displaymath}
	which proves (i).
	Assume now that this infinite word is not shift-periodic. 
	According to \autoref{res: proposition mosse}, there exists $m \geq 2$ such that $\sigma^\infty(a)$ does not contain an $m$-th power.
	The same holds for the prefixes of $\sigma^\infty(a)$, in particular, for all $\sigma^p(a)$, which proves (ii).
	
	\paragraph{}
	Finally, assume that $\sigma^\infty(a)=u^\infty$, where $u$ is a non-trivial primitive word.
	Since $\sigma^\infty(a)$ is fixed by $\sigma$, we obtain that $u^\infty=\sigma(u)^\infty$.
	In particular, $u$ is a prefix of $\sigma^\infty(u)$.
	The substitution $\sigma$ being primitive, $\sigma(u)$ is not shorter than $u$.
	We derive that there exists $w_0 \in \mathcal A^*$ such that $\sigma(u)=uw_0$.
	Hence $u^\infty=(uw_0)^\infty$. 
	\autoref{lem: mosse} shows that there exists $p \in \N$ satisfying $w_0=u^p$.
	Thus $\sigma(u)=u^{p+1}$.
	Remember that $a$ is a prefix of $u$.
	Hence $u$ can be written $u=au'$.
	It follows that the length of $\sigma(u)=aw\sigma(u')$ is larger than the one of $u$. 
	Thus $p +1 \geq 2$, which proves (iii).
\end{proof}

\section{Train-tracks and automorphisms of free groups}

\paragraph{}
In this section, we recollect some material about relative train-track maps.
Details can be found in \cite{BesHan92} where they have been introduced by M.~Bestvina and M.~Handel.
There exist several improvements of relative train-track maps, and we will use here (very few of) improved relative train-track maps introduced by M.~Bestvina, M.~Feighn and M.~Handel in \cite{BesFeiHan00}.

\subsection{Paths and circuits}

\paragraph{}
The graphs that we consider are metric graphs with oriented edges.
By metric graph, we mean a graph equipped with a path metric.
If $e$ is an edge of a graph $G$, then $e\inv$ stands for the edge with the reverse orientation.
The pair $\{e,e\inv\}$ is the \emph{unoriented edge associated to $e$} (or $e\inv$).
By abuse of notation, we will just say the  \emph{unoriented edge $e$} for the pair $\{e,e\inv\}$.
Let $\Theta:\mathcal E\rightarrow \mathcal E$ be the map defined by $\Theta(e)=e\inv$.
Sometimes, it will be useful to consider a subset $\vec{\mathcal E}$ of $\mathcal E$ such that $\vec{\mathcal E}$ and $\Theta(\vec{\mathcal E})$ give rise to a partition of $\mathcal E$ (i.e. we choose a preferred oriented edge for each unoriented edge).
We call such a set $\vec{\mathcal E}$ a {\em preferred set of oriented edges for} $G$.

\paragraph{}
A \emph{path} in a graph $G$ is a continuous locally injective map $\alpha:I\rightarrow G$, where $I=[a,b]$ is a segment of $\R$. 
The \emph{initial point} of $\alpha$ is $\alpha(a)$ and its \emph{terminal point} is $\alpha(b)$; both $\alpha(a)$ and $\alpha(b)$ are the \emph{endpoints} of $\alpha$.
We do not make any difference between two paths which differ from an orientation preserving homeomorphism between their domains. 
A path is \emph{trivial} if its domain is a point.
When the endpoints of $\alpha$ are vertices, $\alpha$ can be viewed as a \emph{path of edges}, i.e. a concatenation of edges: $\alpha= e_1\dots e_p$ where the $e_i$ are edges of $G$ such that the terminal vertex of $e_i$ is the initial vertex of $e_{i+1}$ and $e_{i}\neq e_{i+1}\inv$.
A \emph{circuit} in $G$ is a continuous locally injective map of an oriented circle into $G$.
We do not make any difference between two circuits which differ from an orientation preserving homeomorphism between their domains.
A circuit can be viewed as a concatenation of edges which is well defined up to cyclic permutation.
If $\alpha$ is a path, or a circuit, we denote by $\alpha^{-1}$ the path, or circuit, with the reverse orientation.

\paragraph{}
A continuous map $\alpha:I\rightarrow G$, where $I$ is segment in $\R$, is homotopic relatively to the endpoints to a unique path denoted by $[\alpha]$.
A non homotopically trivial continuous map $\alpha:S^1\rightarrow G$ is homotopic to a unique circuit denoted by $[\alpha]$.

\subsection{Topological representatives}

\paragraph{Marked graphs and topological representatives.}
Let $r \geq 2$.
We denote by $R_r$ the \emph{rose} of rank $r$.
It is a graph with one vertex $\star$ and $r$ unoriented edges.
The fundamental group $\pi_1(R_r,\star)$ is the free group $\free{r}$,
with basis given by a preferred set of oriented edges.
A \emph{marked graph} $(G,\tau)$ (often simply denoted by $G$) is a connected metric graph $G$ having no vertex of valence $1$, equipped with a homotopy equivalence $\tau: R_r\rightarrow G$.
This homotopy equivalence $\tau$ gives an identification of the fundamental group $\pi_1(G)$ with $\free{r}$, well defined up to an inner automorphism.
A \emph{topological representative} of an outer automorphism $\Phi\in\out{\free{r}}$ is a homotopy equivalence $f:G\rightarrow G$ of a marked graph $(G,\tau)$ such that:
\begin{itemize}
	\item $f$ takes vertices to vertices and edges to paths of edges, 
	\item $\tau^-\circ f\circ \tau: R_r\rightarrow R_r$ induces $\Phi$ on $\free{r}=\pi_1(R_r,\star)$, where $\tau^-$ is a homotopy inverse of $\tau$.
\end{itemize}
In particular, the restriction of $f$ to an open edge is locally injective.

\paragraph{Induced map on paths and circuits.}

If $\alpha$ is a path or a circuit in $G$, one defines $f_\#(\alpha)$ as being equal to  $[f(\alpha)]$.

\paragraph{Legal turns.}

For any edge $e$ of $G$, $Df(e)$ denotes the first edge of $f(e)$.
A \emph{turn} is a pair of edges $(e_1,e_2)$ of $G$ which have the same initial vertex.
The turn $(e_1,e_2)$ is \emph{degenerate} if $e_1=e_2$, non degenerate otherwise.
A turn $(e_1,e_2)$ is \emph{legal} if for all $p \in\N$, $((Df)^p(e_1),(Df)^p(e_2))$ is non  degenerate; otherwise, the turn is \emph{illegal}.


\subsection{Lifts}\label{sec:lifts}

Let $f:G\rightarrow G$ be a topological representative of $\Phi\in\out{\free{r}}$.
Let $\tilde{G}$ be the universal cover of $G$.
Galois' covering theory gives a one-to-one correspondence between the set of the  lifts of $f$ to $\tilde{G}$ and the set of automorphisms in the outer class $\Phi$.
More precisely, a lift $\tilde{f}$ of $f$ is in correspondence with the automorphism $\phi\in\Phi$ if:
\begin{equation}\label{eq:phi-tilde f}
\tilde{f}\circ g = \phi(g)\circ \tilde{f}, \;\;\;\;\forall g\in\free{r}
\end{equation}
where the elements of $\free{r}$ are viewed as deck transformations of $\tilde{G}$.

\subsection{Invariant filtrations and transition matrices}

Let $f:G\rightarrow G$ be a topological representative of $\Phi\in\out{\free{r}}$.

\paragraph{Filtration, strata and $k$-legal paths.}

A \emph{filtration} of a topological representative $f:G\rightarrow G$ is a strictly increasing sequence of $f$-invariant subgraphs $\emptyset=G_0 \subset G_1 \subset \dots \subset G_m=G$.
The \emph{stratum of height $k$} denoted by $H_k$ is the closure of $G_k \setminus G_{k-1}$.
The edges \emph{of height $k$} are the edges of $H_k$.
A path \emph{of height $k$} is a path in $G_k$ which crosses $H_k$ non trivially, i.e. its intersection with $H_k$ contains a non-trivial path.
A path (of edges) $\alpha$ is \emph{$k$-legal} if it is a path of $G_k$ and for all subpath $e_1e_2$ of $\alpha$ with $e_1,e_2$ edges of height $k$, the turn $(e_1\inv,e_2)$ is legal.

\paragraph{Transition matrices.}

A \emph{transition matrix} $M_k$ is associated to the stratum $H_k$.
We choose a preferred set of oriented edges for $H_k$: $\vec{\mathcal E}= \{ e_1,\dots,e_\ell \}$
(where $\ell$ is the number of unoriented edges of $H_k$).
The transition matrix $M_k$ of $H_k$ is a square matrix of size $\ell$ whose $(i,j)$-entry is number of times the edge $e_i$ or  the reverse edge $e_i^{-1}$ occur in the path $f(e_j)$.

\paragraph{}
The stratum $H_k$ is \emph{irreducible} when its transition matrix $M_k$ is irreducible.
Let $\lambda_k$ be the PF-eigenvalue of $M_k$ -- see \autoref{sec:primitive matrices}.
If $\lambda_k>1$, then $H_k$ is called an \emph{exponential stratum}.
If $M_k$ is primitive, $H_k$ is said \emph{aperiodic}.
When the stratum $H_k$ is {irreducible} and $\lambda_k=1$, $H_k$ is called a \emph{non-exponential stratum}. 
When $M_k$ is the zero matrix, the stratum $H_k$ is called a \emph{zero stratum}.

\begin{remark}\label{rem:good RTT}
Given a topological representative $f:G\rightarrow G$ and an invariant filtration for $f$, up to refine the filtration, one can always suppose that any stratum is of one of three possible types: exponential, non-exponential or zero.
Moreover, up to replace $f$ by a positive power of $f$, one can assume --~see \cite{BesFeiHan00}~-- that 
\begin{itemize}
	\item each exponential stratum is aperiodic,
	\item each non-exponential stratum $H_k$ consists of a single edge $e$, and that $f(e)=eu$ where $u$ is loop in $G_{k-1}$  based at the endpoint of $e$.
\end{itemize}
\end{remark}

\subsection{A quick review on relative train-track maps}

\paragraph{Relative train-track maps.}

A topological representative $f:G\rightarrow G$ of an outer automorphism $\Phi\in\out{\free{r}}$ with a filtration $\emptyset=G_0 \subset G_1 \subset \dots \subset G_m=G$ is a \emph{relative train-track map} (RTT) if for every exponential stratum $H_k$:
\begin{enumerate}[(RTT-i)]
	\item \label{enu: RRT - Df}
	$Df$ maps the set of edges of height $k$ to itself  (in particular, each turn consisting of an edge of height $k$ and one of height  less than $k$  is legal);
	\item if $\alpha$ is a non-trivial path with endpoints in $H_k\cap G_{k-1}$, then $f_\#(\alpha)$ is a non-trivial path with endpoints in $H_k\cap G_{k-1}$;
	\item for each $k$-legal path $\alpha$, the path $f_\#(\alpha)$ is $k$-legal.
\end{enumerate}

In particular, an edge $e$ of an exponential stratum $H_k$ is $k$-legal.
Theorem 5.12 in \cite{BesHan92} ensures that any outer automorphism $\Phi$ of $\free r$ can be represented by an RTT $f$. 
By replacing if necessary $\Phi$ by a positive power of $\Phi$, one can suppose that $\Phi$ satisfies \autoref{rem:good RTT}.
In addition, we can ask that all the images of vertices are fixed by $f$ (see Theorem 5.1.5 in \cite{BesFeiHan00}). 
We sum up these facts in the following theorem.

\begin{theo}[Bestvina-Handel \cite{BesHan92}, Bestvina-Feighn-Handel \cite{BesFeiHan00}]
\label{thm:bestvina-handel}
	Let $\Phi$ be an outer automorphism of $\free{r}$.
	There exists $p\geq 1$ such that $\Phi^p$ has a topological representative $f:G\rightarrow G$ which is an RTT, with the properties that:
	\begin{itemize}
		\item for all vertices $v$ of $G$, $f(v)$ is fixed by $f$,
		\item every exponential stratum of $f$ is aperiodic.
		\item each non-exponential stratum $H_k$ consists of a single edge $e$, and that $f(e)=eu$ where $u$ is loop in $G_{k-1}$  based at the endpoint of $e$.
	\end{itemize}
\end{theo}

\paragraph{Splittings.}

Let $f:G\rightarrow G$ be a topological representative.
A \emph{splitting} of a path or a circuit $\alpha$ is a decomposition of $\alpha$ as a concatenation of subpaths $\alpha=\alpha_1\alpha_2\ldots\alpha_q$ (with $q\geq 1$ if $\alpha$ is a circuit, and $q\geq 2$ if $\alpha$ is a path) such that for all $p\geq 0$, $f^p_\#(\alpha)=f^p_\#(\alpha_1)f^p_\#(\alpha_2)\dots f^p_\#(\alpha_q)$.
In that case, one writes $\alpha=\alpha_1\cdot\alpha_2\cdot\ldots\cdot\alpha_q$ and $\alpha_1,\alpha_2,\ldots,\alpha_q$ are called the \emph{terms} of the splitting.
A basic, but important, property of RRT is given by the following lemma.

\begin{lemm}[Bestvina-Handel {\cite[Lemma 5.8]{BesHan92}}]
\label{lem:k-legal}
	Let $f:G\rightarrow G$ be an RTT. 
	If $H_k$ is an exponential stratum, and if $\alpha$ is a $k$-legal path, then the decomposition of $\alpha$ as maximal subpaths in $H_k$ or in $G_{k-1}$ is a splitting:
	\begin{equation*}
		\alpha=\alpha_1\cdot\beta_1\cdot\alpha_2\cdot\ldots \cdot\alpha_{q-1}\cdot\beta_{q-1}\cdot\alpha_q,
	\end{equation*}
	where the $\alpha_i$ are paths in $H_k$, the $\beta_i$ are paths in $G_{k-1}$, all non trivial (except possibly $\alpha_1$ and $\alpha_q$).
\end{lemm}

\subsection{Growth of automorphisms of free groups}

\paragraph{}
The growth of an outer automorphism $\Phi\in\out{\free{r}}$ can be detected on an RTT representative. 
For a detailed discussion about the growth of a conjugacy class under iteration of an outer automorphism, one can read G.~Levitt's paper \cite{Lev09}.
For our purpose, we just need the following observations.

\begin{remark}\label{rem:pol/exp}
	Let $\Phi \in \out{\free r}$.
	\begin{enumerate}
		  \item \label{item:pol or exp} 
		  $\Phi$ has either polynomial growth, or exponential growth. 
		  \item \label{item: >1 exp}  
		  Moreover, $\Phi$ has exponential growth if and only if one (hence any) RRT $f:G\rightarrow G$ representing $\Phi$ has at least one exponential stratum.
	\end{enumerate}
\end{remark}

\section{Reductions of \autoref{theo:main}}
\label{sec:reductions}

\paragraph{}In this section we explain how to reduce our main theorem to an easier statement.
First, note that given an outer automorphism $\Phi$ of the free group, $\Phi$ has exponential (\resp polynomial) growth if and only if for every $p \in \N^*$, so has $\Phi^p$.
In particular, to prove \autoref{theo:main},  $\Phi$ can be replaced  by some positive power  of $\Phi$.
It will be advantageous to do so, since it allows us to use relative train-track maps with better properties (see \autoref{thm:bestvina-handel}).

\subsection{Polynomially growing automorphisms}
\label{sec:pol. growing}

\paragraph{}
Arguing by induction on the rank $r$ of $\free{r}$, the first author handled the case of polynomially growing automorphisms.

\begin{prop}[Coulon {\cite[Theorem 1.6]{Cou10a}}]\label{prop:pol}
If $\Phi\in\out{\free{r}}$ is polynomially growing, then for all positive integers $n$, $\Phi$ induces an outer automorphism of finite order of $\burn{r}{n}$.
\end{prop}

\begin{remark}
	The same proof actually gives a quantitative bound for the order of $\Phi$ in $\out{\burn rn}$.
	If $\Phi$ is an outer polynomially growing automorphism of $\free r$, then $\Phi^{p(r,n)}$ induces a trivial outer automorphism of $\burn rn$ where 
	\begin{equation*}
		p(r,n)  = n^{2(2^{r-1}-1)}.
	\end{equation*}
\end{remark}

\begin{example}
\label{exa: dehn-twist}
A particular case of polynomially growing automorphisms is given by the automorphisms of $\free 2$ induced by a Dehn-twist on a punctured torus.
For instance the automorphism $\phi$ defined by $\phi(a) = a$ and $\phi(b) = ba$.
Here $\phi^n$ is trivial in $\aut{\burn rn}$.
\end{example}

In view of \autoref{rem:pol/exp}~(\ref{item:pol or exp}) and \autoref{prop:pol}, we see that \autoref{theo:main} is a consequence of the following proposition.

\begin{prop}\label{res: reduction - poly}
If $\Phi\in\out{\free{r}}$ has exponential growth, then there exist $\kappa, n_0\in\N$ such that for all odd integers $n\geq n_0$,  $\Phi$ induces an outer automorphism of $\burn r{\kappa n}$ of infinite order.
\end{prop}

\subsection{Passing to a finite index subgroup}
\label{sec: finite index subgroup}

\paragraph{}
Let $\Phi$ be an exponentially growing outer automorphism of $\free r$.
By replacing if necessary $\Phi$ by a power of $\Phi$ we can assume that $\Phi$ is represented by an RTT $f:G\rightarrow G$ with a filtration $\emptyset=G_0 \subset G_1 \subset \dots \subset G_m=G$, satisfying the properties of \autoref{thm:bestvina-handel}.
We denote by $H_k$ the stratum of height $k$.
Let $e$ be an edge of an exponential stratum $H_k$.
According to \autoref{lem:k-legal}, $f(e)$ can be split as follows
\begin{equation*}
	f(e) = \alpha_1 \cdot \beta_1 \cdot \alpha_2 \cdot \dots \cdot \alpha_{q-1} \cdot\beta_{q-1} \cdot \alpha_q,
\end{equation*}
where the $\alpha_i$'s are non trivial paths contained in $H_k$ and the $\beta_i$'s non trivial paths contained in $G_{k-1}$.
We denote by $\mathcal P_e$ the set $\{\beta_1, \dots, \beta_{q-1}\}$.
Let $\mathcal P$ be the union of $\mathcal P_e$ for all edges $e$ belonging to an exponential stratum.
Note that $\mathcal P$ is finite.

\paragraph{}
Let $p \in \N$ and $e$ be an edge of the exponential stratum $H_k$.
Recall that $G_{k-1}$ is $f$-invariant.
Thus if $\beta$ is a maximal subpath of $f^p_\#(e)$ contained in $G_{k-1}$ then it is the image by some (possibly trivial) power of $f_\#$ of a path in $\mathcal P$.
Moreover we assumed that the image by $f$ of any vertex of $G$ is fixed by $f$.
Hence if in addition $\beta$ is a loop, there exists a path $\beta'$ in $\mathcal P \cup f_\#(\mathcal P)$ which is a loop such that $\beta$ is the image of $\beta'$ by a (possibly trivial) power of $f_\#$.
Since $\free r$ is residually finite, there exists a finite-index normal subgroup $\mathbf H$ of $\free r$ with the following property.
For every path $\beta \in \mathcal P \cup f_\#(\mathcal P)$, if $\beta$ is a loop then the conjugacy class of $\free r$ that it represents does not intersect $\mathbf H$.

\paragraph{}
Recall that $\tilde G$ stands for the universal cover of $G$.
Let us fix a base point $x_0$ in $G$.
The fundamental group $\free r = \pi_1(G,x_0)$ can therefore be identified with the deck transformation group acting on the left on $\tilde G$.
We fix a lift $\tilde f : \tilde G \rightarrow \tilde G$ of $f$.
It determines an automorphism $\phi$ in the outer class of $\Phi$ such that for every $g \in \free r$,
\begin{equation}
\label{eqn: passing finite index - lift auto}
	\phi(g) \circ \tilde f = \tilde f \circ g.
\end{equation}
There are only finitely many subgroups of $\free r$ of a given index.
Thus there exists an integer $q$ such that $\phi^q(\mathbf H) = \mathbf H$.
Consequently the intersection $\mathbf L = \bigcap_{p \in \Z} \phi^p(\mathbf H)$ is also a normal finite-index subgroup of $\free r$.
Moreover it is $\phi$-invariant.
As a consequence of the LERF property, it implies that the restriction of $\phi$ to $\mathbf L$ is an automorphism.
See for instance Lemma~6.0.6 in \cite{BesFeiHan00}, attributed there to Scott.

\paragraph{}
We now denote by $\kappa$ the index of $\mathbf L$ in $\free r$. 
Let $\hat G$ be the space $\hat G = \mathbf L \backslash \tilde G$ and $\rho \colon \hat G \rightarrow G$ be the natural projection induced by $\tilde G \rightarrow G$.
The group $\free r$ still acts on the left on $\hat G$ and $\mathbf L$ is the kernel of this action.
The map $\tilde f$ induces a map $\hat f \colon \hat G \rightarrow \hat G$ such that $\rho \circ \hat f = f \circ \rho$. Moreover, according to (\ref{eqn: passing finite index - lift auto}), for every $g \in \free r$, 
\begin{equation}
\label{eqn: finite cover - lift auto}
	\phi(g) \circ \hat f = \hat f \circ g.
\end{equation}

\begin{lemm}
\label{res: lifting a train track in a train track}
	The map $\hat f : \hat G \rightarrow \hat G$ admits a filtration which makes $\hat f$ be an RRT representing the outer class of $\phi$ restricted to $\mathbf L$.
	Moreover, for every exponential stratum $\hat H$ of $\hat G$ there exists $k \in \intvald 1m$ such that
	\begin{enumerate}
		\item $H_k$ is an exponential stratum of $G$,
		\item $\hat H$ is contained in $\rho^{-1}(H_k)$ and
		\item $\hat f$ sends $\hat H$ in $\hat H \cup \rho^{-1}(G_{k-1})$.
	\end{enumerate}
\end{lemm}

\begin{proof}
We observe that, by construction, $\hat f : \hat G \rightarrow \hat G$ is a topological representative of $\phi$ restricted to $\mathbf L$, and $\emptyset=\rho^{-1}(G_0) \subset \rho^{-1}(G_1) \subset \dots \subset \rho^{-1}(G_m)=\hat G$ is an invariant filtration for $\hat f$.
We are going to define a finer filtration $(\hat G_{k,j})$ where the pairs $(k,j)$ are endowed with the lexicographical order such that for every $k \in \intvald 1m$ we have 
\begin{equation*}
	\rho^{-1}(G_{k-1}) \subset \hat G_{k,1} \subset \hat G_{k,2} \subset \ldots \subset \hat G_{k,s} = \rho^{-1}(G_k).
\end{equation*}
Let $k \in \intvald 0m$.
We focus on the stratum $H_k$ of height $k$ of $G$.
We distinguish three cases.
\begin{enumerate}
	\item If $H_k$ is a zero stratum, we just put $\hat G_{k,1}=\rho^{-1}(G_k)$.
	Since $\rho \circ \hat f = f \circ \rho$, we have $\hat f(\hat G_{k,1})\subseteq\rho^{-1}(G_{k-1})$.
	Therefore the associated stratum is a zero stratum.
	
	\item If $H_k$ is a non-exponential stratum, it consists of a single edge $e$ with $f(e)=eu$ where $u$ is a loop in $G_{k-1}$. 
	Then $\rho^{-1}(e)$ is a collection of $\kappa$ edges: $\hat e_1,\dots,\hat e_\kappa$ (recall that $\kappa$ is the index of $\mathbf L$ in $\free{r}$). 
	For every $j \in \intvald 1\kappa$, we put $\hat G_{k,j} = \{\hat e_1,\dots,\hat e_j\} \cup \rho^{-1}(G_{k-1})$.
	Since $\rho \circ \hat f = f \circ \rho$, we get  that $\hat f(\hat e_j)=\hat e_j \hat u_j$ where $\hat u_j$ is a path in $\rho^{-1}(G_{k-1})$. 
	In particular, $\hat f$ let invariant the filtration 
	\begin{equation*}
		\rho^{-1}(G_{k-1}) \subset \hat G_{k,1} \subset \ldots \subset \hat G_{k,\kappa} = \rho^{-1}(G_k).
	\end{equation*}
	Moreover each stratum coming from this filtration consists of a single edge $\hat e_j$ and its transition matrix is the scalar matrix $1$.
	Thus it is a non-exponential stratum.
	
	\item Assume now that $H_k$ is an exponential stratum.
	We define a binary relation on $\rho^{-1}(H_k)$.
	Given two edges $\hat e_1$ and $\hat e_2$ we say that $\hat e_1 \sim \hat e_2$ if there exists $p \in \N$ such that $\hat e_1$ or $\hat e_1^{-1}$ is an edge of $\hat f^p(\hat e_2)$.
	This relation is reflexive and transitive.
	We claim that it is an equivalence relation.
	Let $\hat e_1$ and $\hat e_2$ be two edges of $\rho^{-1}(H_k)$ such that $\hat e_1 \sim \hat e_2$.
	We want to prove that $\hat e_2 \sim \hat e_1$. 
	By definition of our relation there exists $p \in \N$ such that $\hat e_1$ or $\hat e_1^{-1}$ is an edge of $\hat f^p(\hat e_2)$.
	For simplicity we assume that $\hat e_1$ belongs to $\hat f^p(\hat e_2)$.
	The other case works in the same way.
	We write $e_1 = \rho(\hat e_1)$ and $e_2 = \rho (\hat e_2)$ for their respective images in $G$.
	Since the stratum $H_k$ is aperiodic there exists $q \in \N$ such that $e_2$ or $e_2^{-1}$ is an edge of $f^q(e_1)$.
	For simplicity we assume that $e_2$ is an edge $f^q(e_1)$.
	Since $\rho \circ \hat f = f \circ \rho$ there exists a preimage of $e_2$ in $\hat G$ which is an edge of $\hat f^q(\hat e_1)$.
	Thus there exists $u \in \free r$ such that $u \cdot\hat e_2$ is an edge of $\hat f^q(\hat e_1)$.
	We now prove by induction that for every $\ell \in \N$, $u_\ell \cdot \hat e_2$ is an edge of $\hat f^{\ell(p+q) +q}(\hat e_1)$, where
	\begin{equation*}
		u_\ell  = \phi^{\ell(p+q)}(u) \cdots \phi^{p+q}(u)u.
	\end{equation*}
	If $\ell = 0$ the statement follows from the definition of $u$.
	Assume that it is true for $\ell \in \N$, i.e. $u_\ell\cdot\hat e_2$ is an edge of $\hat f^{\ell(p+q) +q}(\hat e_1)$.
	Using (\ref{eqn: finite cover - lift auto}) we get that $\phi^p(u_\ell)\cdot\hat f^p(\hat e_2) = \hat f^p(u_\ell\hat e_2)$ is a subpath of $\hat f^{(\ell+1)(p+q)}(\hat e_1)$.
	In particular $\phi^p(u_\ell)\cdot \hat e_1$ lies in $\hat f^{(\ell+1)(p+q)}(\hat e_1)$.
	With a similar argument we get that $\phi^{p+q}(u_\ell)u\cdot \hat e_2$ lies in $\hat f^{(\ell+1)(p+q)+q}(\hat e_1)$.
	However 
	\begin{equation*}
		 \phi^{p+q}(u_\ell)u = \phi^{(\ell+1)(p+q)}(u) \cdots \phi^{p+q}(u)u = u_{\ell+1}.
	\end{equation*}
	Thus the statement holds for $\ell+1$, which completes the proof of the induction.
	
	\paragraph{}
	Since $\mathbf L$ has finite index in $\free r$, there exist $\ell \in \N$ and $t\in \N^*$ such that 
	$u_\ell$ and $u_{\ell+t}$ are in the same $\mathbf L$-coset, 
	i.e. $u_{\ell+t} u_\ell^{-1} \in \mathbf L$.
	However
   	\begin{equation*}
		u_{\ell+t} u_\ell^{-1} = \phi^{(\ell+t)(p+q)}(u) \cdots \phi^{(\ell+1)(p+q)}(u)
		= \phi^{(\ell+1)(p+q)}(u_{t-1}).
	\end{equation*}
	Since $\mathbf L$ is $\phi$-invariant, we derive that $u_{t-1}$ belongs to $\mathbf L$.
	Recall that $\mathbf L$ is the kernel of the action of $\free r$ on $\hat G$, 
	hence $u_{t-1}\cdot \hat e_2 = \hat e_2$.
	On the other hand, $u_{t-1}\cdot \hat e_2$ is an edge of $\hat f^{(t-1)(p+q)+q}(\hat e_1)$.
	Consequently $\hat e_2 \sim \hat e_1$, which completes the proof of our claim.

	\paragraph{}
	We denote by $\hat H_{k,1}, \cdots, \hat H_{k,s}$ the equivalence classes for the relation~$\sim$.
	For every $j \in \intvald 1s$, we put $\hat G_{k,j} = \hat H_{k,1} \cup \dots \cup \hat H_{k,j} \cup \rho^{-1}(G_{k-1})$.
	By construction, the filtration 
	\begin{equation*}
		\rho^{-1}(G_{k-1}) \subset \hat G_{k,1} \subset \ldots \subset \hat G_{k,s}= \rho^{-1}(G_k)
	\end{equation*}
	is $\hat f$-invariant.
	Moreover the strata $\hat H_{k,1}, \dots, \hat H_{k,s}$ associated to this filtration are irreducible.
	We claim that they are exponential.
	Let $j \in \intvald 1s$. 
	Let $M_{k,j}$ be the transition matrix of $\hat H_{k,j}$.
	It is known that if the PF-eigenvalue of $M_{k,j}$ is $1$, then $M_{k,j}$ is a permutation matrix. 
	Thus there exist an edge $\hat e \in \hat H_{k,j}$ and a positive integer $p$ such that $\hat e$ is the only edge of $\rho^{-1}(H_k)$ in $\hat f^p(\hat e)$.
	In particular, if $e$ stands for $e = \rho (\hat e)$ we get that $e$ is the only edge of $H_k$ in $f^p(e)$.
	This contradicts the fact that $H$ is aperiodic.
	Hence the PF-eigenvalue of $M_{k,j}$ is larger than $1$ and the stratum $\hat H_{k,j}$ is exponential.
\end{enumerate}
Finally, recall that  $f$ satisfies properties (RTT-i), (RTT-ii) and (RTT-iii).
It follows from $\rho \circ \hat f = f \circ \rho$, that $\hat f$ also satisfies properties (RTT-i), (RTT-ii) and (RTT-iii).
\end{proof}

\begin{lemm}
\label{res: no loop in the lifted RTT}
	For every edge $\hat e$ in an exponential stratum $\hat H$ of $\hat G$, for every $p \in \N$, every maximal subpath of $\hat f^p_\#(\hat e)$ that does not cross $\hat H$ is not a loop.
\end{lemm}

\begin{proof}
	Let $\hat e$ be an edge of an exponential stratum  $\hat H$ of $\hat G$.
	Let $p \in \N$.
	Let $\hat \beta$ be  a maximal subpath of $\hat f^p_\#(\hat e)$ that does not cross $\hat H$.
	By \autoref{res: lifting a train track in a train track}, there exists $k \in \intvald 1m$ such that $H_k$ is an exponential stratum of $G$, $\hat H$ is contained in $\rho^{-1}(H_k)$ and $\hat f(\hat H)$ is a subset of $\hat H \cup \rho^{-1}(G_{k-1})$.
	We denote by $e$ the image of $\hat e$ by $\rho$.
	It belongs to $H_k$.
	Since $\rho$ is a continuous locally injective map, $\hat f^p_\#(\hat e) $ is a lift of $f^p_\#(e)$.
	It follows that $\beta = \rho(\hat \beta)$ is a maximal subpath of $f^p_\#(e)$ contained in $G_{k-1}$.
	If $\beta$ is not a loop, neither is $\hat \beta$.
	Therefore we can assume that $\beta$ is a loop in $G$.
	By construction of $\mathcal P$, there exists a loop $\beta'$ in $\mathcal P \cup f_\#(\mathcal P)$ such that $\beta$ is the image of $\beta'$ by some power of $f_\#$.
	However, by definition, the conjugacy class of $\free r$ represented by $\beta'$ does not intersect $\mathbf L\subset \mathbf H$.
	Since $\mathbf L$ is $\phi$-invariant, neither does the conjugacy class of $\free r$ represented by $\beta$.
	Thus its lift $\hat \beta$ in $\hat G$ cannot be a loop.
\end{proof}

\begin{lemm}
\label{res: passing to a finite index subgroup}
	Let $n$ be an integer.
	Recall that $\kappa$ is the index of $\mathbf L$ in $\free r$.
	If $\Phi$ induces an outer automorphism of finite order of $\burn r{\kappa n}$ then its restriction to $\mathbf L$ induces an outer automorphism of finite order of $\mathbf L /\mathbf L^n$.
\end{lemm}

\begin{proof}
	According to \autoref{rem: confusion aut out}, the image of $\phi$ in $\aut{\burn r{\kappa n}}$ has finite order.
	Hence there exists $p \in \N$ such that for every $g \in \free r$, $\phi^p(g)g^{-1}$ belongs to $\free r^{\kappa n}$.
	However, $\mathbf L$ has index $ \kappa$ in $\free r$.
	It follows that for every $g \in \free r$, $g^\kappa$ lies in $\mathbf L$. 
	In particular, $\free r^{\kappa n}$ is a subset of $\mathbf L^n$.
	Consequently for every $g \in \mathbf L$, $\phi^p(g)g^{-1}$ belongs to $\mathbf L^n$.
	It exactly means that, as an automorphism of $\mathbf L / \mathbf L^n$, $\phi^p$ is trivial.
	Hence the restriction of $\Phi$ to $\mathbf L$ induces an automorphism of finite order of $\mathbf L /\mathbf L^n$.
\end{proof}

\autoref{res: reduction - poly} becomes a consequence of the following result.

\begin{prop}
\label{res: reduction - finite index subgroup}
	Let $\Phi\in\out{\free{r}}$ be an outer automorphism represented by an RTT $f:G\rightarrow G$.
	Assume that for every edge $e$ in an exponential stratum $H$, for every $p \in \N$, every maximal subpath of $f^p_\#(e)$ that does not cross $H$ is not a loop.
	Then there exists $n_0\in\N$ such that for all odd integers $n\geq n_0$,  $\Phi$ induces an outer automorphism of $\burn{r}{n}$ of infinite order.
\end{prop}

\subsection{Automorphisms with only one exponential stratum}
\label{sec:one exp. stratum}

\paragraph{}
The following lemma is proved by the first author in \cite{Cou10a} using the structure of free products.

\begin{lemm}[Coulon {\cite[Lemma 1.9]{Cou10a}}]
\label{res:injective burnside map}
	Let $n$ be an integer.
	Let $\phi$ be an automorphism of $\free r$ which stabilizes a free factor $\mathbf H$.
	We assume that $\phi$ induces an automorphism of finite order of $\burn rn$.
	Then, the restriction of $\phi$ to $\mathbf H$ also  induces an automorphism of finite order of $\mathbf H/\mathbf H^n$.
\end{lemm}

\paragraph{}
Let $\Phi\in\out{\free{r}}$ be an exponentially growing outer automorphism,  and let $f:G\rightarrow G$ be an RTT representing $\Phi$ with a filtration $\emptyset=G_0 \subset G_1 \subset \dots \subset G_m=G$.
By \autoref{rem:pol/exp}~(\ref{item: >1 exp}), $f$ has at least one exponential stratum.
We assume that $f$ satisfies the additional assumption of \autoref{res: reduction - finite index subgroup}, i.e. for every edge $e$ in an exponential stratum $H$, for every $p \in \N$, every maximal subpath of $f^p_\#(e)$ that does not cross $H$ is not a loop.
By replacing if necessarily $\Phi$ by a power of $\Phi$, we can assume that the exponential strata of the RTT are aperiodic.
Note that this operation does not affect the graph $G$.
However one might need to refine the filtration of the RTT.
In particular the RTT does not necessarily satisfy anymore the additional assumption of \autoref{res: reduction - finite index subgroup}.
Nevertheless for every edge $e$ in the \emph{lowest} exponential stratum $H_k$  for every $p \in \N$, every maximal subpath of $f^p_\#(e)$ that does not cross $H_k$ is not a loop.

\paragraph{}
Let $G'$ be the connected component of the graph $G_k$  which contains $H_k$.
We assert that $G'$ is $f$-invariant, i.e. $f(G')\subseteq G'$.
Indeed, $G'\cap f(G')$ is non-empty (it contains $H_k$) and $f(G')$ is connected.
Let $\mathbf H$ be the free factor of $\free{r}$ defined by $G'\subseteq G$, and let $\Psi\in\out{\mathbf H}$ be the outer automorphism induced by the restriction $f'=\restriction f{G'}:G'\rightarrow G'$.
We note that $\restriction f{G'}:G'\rightarrow G'$ is an RTT representing $\Psi$, which has exactly one exponential stratum, namely $H_k$, which is aperiodic and the top stratum of $f'$.
In particular $\Psi$ has exponential growth.

\begin{lemm}
\label{lemm:restriction}
	If $\Psi$ induces an outer automorphism of $\mathbf H/\mathbf H^n$ of infinite order, then $\Phi$ also induces an outer automorphism of $\burn{r}{n}$ of infinite order.
\end{lemm}

\begin{proof}
	There exists an automorphism $\phi$ in the class of $\Phi$ which stabilizes $\mathbf H$.
	Assume that $\Phi$ induces an outer automorphism of $\burn rn$ of finite order.
	In particular the image of $\phi$ in $\aut{\burn rn}$ has finite order (see \autoref{rem: confusion aut out}).
	It follows from the previous lemma that the restriction to $\mathbf H$ of $\phi$ (and thus $\Psi$) induces an automorphism (outer automorphism) of finite order of $\mathbf H/\mathbf H^n$.
\end{proof}

It follows from our discussion that \autoref{res: reduction - finite index subgroup} is a consequence of the following statement.

\begin{prop}
\label{res: reduction - one exponential stratum}
	Let $\Phi\in\out{\free{r}}$ be an outer automorphism represented by an RTT $f:G\rightarrow G$ with exactly one exponential stratum $H$, which is aperiodic and the top stratum of $f$.
	Assume that for every edge $e$ in $H$, for every $p \in \N$, every  maximal subpath of $f^p_\#(e)$ that does not cross $H$ is not a loop.
	Then there exists $n_0\in\N$ such that for all odd integers $n\geq n_0$,  $\Phi$ induces an outer automorphism of $\burn{r}{n}$ of infinite order.
\end{prop}

\section{Tracking powers}
\label{sec:no big powers}

\paragraph{}The next two sections are dedicated to the proof of \autoref{res: reduction - one exponential stratum}.
As we explained in the introduction, the goal is to understand in which extend a periodic path can appear in the orbit of a circuit under the iteration of the train-track map.
This is the purpose of this section.

\paragraph{}
The general strategy is the following. 
We consider an outer automorphism $\Phi$ represented by an RTT $f \colon G \rightarrow G$ with a single exponential stratum $H$ which is aperiodic.
Then, we fix an edge $e_\bullet$ in $H$.
For every $p \in \N$, we look at the path obtained by removing from $f_\#^p(e_\bullet)$ all the edges which are not in $H$.
This sequence can be interpreted as the orbit of $e_\bullet$ under a substitution over the set of oriented edges of $H$ (\autoref{res: substitution and RTT coincide on the red}).
It follows from the aperiodicity of $H$ that this substitution is primitive.
Therefore we would like to apply \autoref{res: no big power in the orbit of the substitution}.
We need to rule out first the case of an infinite shift-periodic word, though (\autoref{res: no big power in the orbit of the substitution}~(\ref{enu: no big power in the orbit of the substitution - not shift-preiodic})). 
The dynamic of the substitution is not sufficient to conclude here. 
\autoref{rem: primitive shift periodic substitution} provides indeed  an example of a primitive substitution $\sigma$ with an infinite shift-periodic fixed point.
However the particularity of this example is that $\sigma$ does not represent an automorphism of $\free 3$.
Our proof (see \autoref{res: sigma(e) not shift periodic}) strongly uses the fact that the substitution we are looking at comes from an automorphism of the free group.

\paragraph{}
From now on, $\Phi$ denotes an outer automorphism of ${\free{r}}$  which can be represented by an RTT $f:G\rightarrow G$ with exactly one exponential stratum $H$.
Moreover, $H$ is aperiodic and the top stratum of $f$.
We denote by $\mathcal E$ the set of all the oriented edges of $H$.
In addition, we assume that for every $e\in \mathcal E$, for every $p\in\N$, every maximal subpath of $f^p_\#(e)$ that does not cross $H$ is not a loop.

\paragraph{}
By replacing if necessary $\Phi$ by a power of $\Phi$ we can assume that for every vertex $v$ of $G$, $f(v)$ is fixed by $f$ and that there exists $e_\bullet \in \mathcal E$ such that $Df(e_\bullet) = e_\bullet$.
Note that this operation does not affect the graph $G$ or the exponential stratum. 
In particular, $H$ is still the only exponential stratum of $f$.
It is aperiodic and the top stratum.
By choice of $e_\bullet$, $f$ fixes the initial vertex $x_0$ of $e_\bullet$.
Thus it naturally defines an automorphism $\phi\in\aut{\pi_1(G, x_0)}$ in the outer class $\Phi$: if $g$ is an element of $\pi_1(G, x_0)$ represented by a loop $\alpha$ based at $x_0$, then $\phi(g)$ is the homotopy class of $f(\alpha)$ (relative to $x_0$).

\subsection{The yellow-red decomposition.}

\paragraph{}
We refer to the edges of $H$ as \emph{red edges} and to the edges of  $G\setminus H$ as \emph{yellow edges}.
Recall that $\tilde{G}$ denotes the universal cover of $G$.
An edge of $\tilde{G}$ can be labelled by the edge of $G$ of which it is the lift.
In particular its color is given by the color of its label.

\paragraph{}
A $k$-legal path of $G$ (where $k$ is the height of $H$) will be call a \emph{red-legal} path.
A path (in $G$ or in $\tilde{G}$) is a \emph{yellow path} if it only crosses yellow edges.
\emph{Red paths} are defined in the same way.
Any path $\alpha$ (in $G$ or in $\tilde{G}$) can be decomposed as a concatenation of maximal yellow and red subpaths:
$\alpha=\alpha_1\dots \alpha_q$, where $\alpha_i$ ($1\leq i\leq q$) is a non trivial subpath of $\alpha$ which is either yellow or red, and  for all $i \in \intvald 1{q-1}$, $\alpha_i$ and $\alpha_{i+1}$ have not the same color.
According to \autoref{lem:k-legal}, this decomposition is a splitting of $\alpha$.

\paragraph{The red word associated to a path.}

We associate to any path of edges $\alpha$ in $G$ or $\tilde G$ a word $\Red \alpha$ over the alphabet $\mathcal E$. 
As a path of edges, $\alpha$ is labelled by a word over the alphabet that consists of all oriented edges of $G$.
The word $\Red{\alpha}$ is obtained from this word by removing all the letters corresponding to yellow edges.
We stress on the fact  that if $\alpha$ is a reduced path, then $\Red{\alpha}$ is not, in general, a reduced word.

\subsection{The induced substitution on red edges.}

\paragraph{Definition and first properties.}

We associate to the RTT $f$ a substitution $\sigma$ on $\mathcal E$ called the \emph{induced substitution}.
It is defined as follows.
\begin{equation*}
	\text{For every } e \in \mathcal E, \quad \sigma(e)=\Red{f(e)}.
\end{equation*}

\begin{lemm}
\label{res: substitution and RTT coincide on the red}
	Let $\alpha$ be a red-legal path in $G$.
	For all $p\in\N$, we have
	\begin{equation*}
		\Red{f_\#^p(\alpha)}=\sigma^p(\Red{\alpha}).
	\end{equation*}
\end{lemm}

\begin{proof}
We consider a decomposition of $\alpha$ as $\alpha=\alpha_1e_1\alpha_2e_2\dots \alpha_q e_q \alpha_{q+1}$ where each $e_i \in \mathcal E$ is a red edge, and each $\alpha_i$ is a (possibly trivial) yellow subpath.
In particular, $\Red{\alpha}=e_1e_2\dots e_q$.
The path $\alpha$ being red-legal, \autoref{lem:k-legal} leads to
\begin{eqnarray*}
	f_\#(\alpha) & = & f_\#(\alpha_1e_1\alpha_2e_2\dots \alpha_q e_q \alpha_{q+1})\\
 	& = & f_\#(\alpha_1)f(e_1)f_\#(\alpha_2)f(e_2)\dots f_\#(\alpha_q) f(e_q) f_\#(\alpha_{q+1}).
\end{eqnarray*}
However, $f$ sends yellow edges to yellow paths.
We deduce that
\begin{eqnarray*}
	\Red{f_\#(\alpha)} & = &  \Red{f(e_1)}\Red{f(e_2)}\dots \Red{f(e_q)} \\
	& = & \sigma(e_1)\sigma(e_2)\dots \sigma(e_q)
 	 = \sigma(e_1 e_2\dots e_q)
 	= \sigma(\Red{\alpha})
\end{eqnarray*}
The image by $f_\#$ of a red-legal path is still a red-legal path.
Therefore for all $p \in \N$,
\begin{displaymath}
	\Red{f^{p+1}_\#(\alpha)} = \Red{f_\#\left(f^p_\#(\alpha)\right)} = \sigma \left(\Red{f^p_\#(\alpha)}\right).
\end{displaymath}
The result follows by induction on $p$.
\end{proof}

\paragraph{Primitivity of the induced substitution.}
The material of this paragraph is widely inspired by the work of P.~Arnoux and all. \cite[Section 3]{Arnoux:2006wt}.
Recall that $\Theta : \mathcal E \rightarrow \mathcal E$ is the map which sends $e$ to $e^{-1}$.
We extend $\Theta$ to the free monoid $\mathcal E^*$ in the following way.
Let $w$ be an element of $\mathcal{E}^*$.
By definition, it can be written $w =e_1 e_2\dots e_q$ where $e_i\in \mathcal{E}$.
We put $\Theta(w) = e_q\inv\dots e_2\inv e_1\inv$.
It defines an involution of $\mathcal{E}^*$ called the \emph{flip map}.
Moreover, we observe that for all edges $e \in \mathcal E$, $\sigma\circ \Theta(e) = \Theta \circ \sigma(e)$.
Thus $\sigma$ and $\Theta$ commute on $\mathcal E^*$.
The substitution $\sigma$ is said to be \emph{orientable} with respect to a subset $\vec {\mathcal E}$ of $\mathcal E$ if 
\begin{enumerate}[(i)]
\item \label{item: prefered} $\vec {\mathcal E}$ and $\Theta(\vec {\mathcal E})$ make a partition of $\mathcal E$,
\item $\sigma(\vec {\mathcal E})\subset \vec {\mathcal E}^*$.
\end{enumerate}
Note that (\ref{item: prefered}) just says that $\vec {\mathcal E}$ is a preferred set of oriented edges for $H$.
In that case, $\sigma$ induces a substitution of $\vec {\mathcal E}^*$, that we still denote by $\sigma$.

\paragraph{} By assumption the red stratum $H$ of $f$ is aperiodic.
In other words its transition matrix $M$ is primitive.
Applying \cite[Proposition 3.7]{Arnoux:2006wt}, we know that
\begin{itemize}
	\item either $\sigma$ is not orientable and then $\sigma$ is a primitive substitution  on the alphabet $\mathcal{E}$, 
	\item or there exists a subset $\vec {\mathcal E}$ of $\mathcal E$ such that $\sigma$ is orientable with respect to $\vec {\mathcal E}$ and then $\sigma$ induces a primitive substitution  on the alphabet  $\vec {\mathcal E}$.
\end{itemize}
Thus in both cases, there exists a subset $\mathcal E_\bullet$ of $\mathcal E$ containing $e_\bullet$ such that $\sigma(\mathcal E_\bullet) \subset \mathcal E_\bullet^*$ and the substitution $\sigma : \mathcal E_\bullet^* \rightarrow \mathcal E_\bullet^*$ is primitive.

\subsection{A red word without large powers}

\paragraph{The infinite red word $\sigma^\infty(e_\bullet)$.}

Recall that $e_\bullet$ is a red edge of $\mathcal E$ that has been chosen in such a way that $Df(e_\bullet)=e_\bullet$.
Because the red stratum is aperiodic, $f(e_\bullet)=e_\bullet \cdot\alpha$ where $\Red{\alpha}$ is non trivial.
In particular $e_\bullet$ is a prefix of $\sigma(e_\bullet)$.
According to \autoref{res: no big power in the orbit of the substitution} the sequence $(\sigma^p(e_\bullet))$ converges to an infinite word $\sigma^\infty(e_\bullet)$ of $\mathcal E_\bullet^\N$. 
Note that $f(e_\bullet)=e_\bullet \cdot\alpha$ is a splitting.
Hence for every $p \in \N$,
\begin{equation*}
	f_\#^p(e_\bullet) = e_\bullet \cdot \alpha \cdot f_\#(\alpha) \cdot \ldots \cdot f^{p-1}_\#(\alpha). 
\end{equation*}
Hence $(f_\#^p(e_\bullet))$ also converges to an infinite path
\begin{equation*}
	f_\#^\infty(e_\bullet) = e_\bullet \cdot \alpha \cdot f_\#(\alpha) \cdot \ldots \cdot f^p_\#(\alpha)\cdot \ldots
\end{equation*}

\begin{prop}
\label{res: sigma(e) not shift periodic}
	The infinite word $\sigma^\infty(e_\bullet)$ is not shift periodic.
\end{prop}

\paragraph{}This proof combines a dynamical argument ($\sigma$ is a primitive substitution) and a group theoretical one ($\phi$ is an automorphism of $\free r$).
Let us sketch first the main steps.
We assume that the proposition is false.
This means that if we restrict our attention to the red edges, the path $f_\#^\infty(e_\bullet)$ is periodic.
We construct from $G$ a colored graph $\Gamma$ on which $f_\#^\infty(e_\bullet)$ coils up.
More precisely its fundamental group $\mathbf H$ can be decomposed as a free product $\mathbf H =  \mathbf L*\langle h \rangle$, where $\mathbf L$ is generated by conjugates of yellow loops and $h$ is represented by a loop $\hat \gamma$ with the following property.
If we collapse all the yellow edges of $\Gamma$ we obtain a simple (red) loop which is exactly the image of $\hat \gamma$ by the same operation.
Moreover, this red loop is the period of the red word associated to $f_\#^\infty(e_\bullet)$ (see \autoref{fig: graph Gamma}).
We show that the RTT $f$ induces a homotopy equivalence $\hat f \colon \Gamma \rightarrow \Gamma$ that catches two conflicting features of $\Phi$:
\begin{enumerate}
	\item Since the stratum $H$ is exponential, $\hat f$ should increase the length of the red word associated to $\hat \gamma$ (see \autoref{res: lifting f in Gamma}).
	\item The yellow components of $G$ are invariant under $f$. It follows that the automorphism of $\mathbf H$ induced by $\hat f$ sends $h$ to $gh^{\pm1}$, where $g$ belongs to the normal subgroup generated by $\mathbf L$.
\end{enumerate}
The key fact is that these two properties can be observed in the abelianization of $\mathbf H$ which leads to a contradiction.

\begin{proof}[Proof of \autoref{res: sigma(e) not shift periodic}]
Assume that $\sigma^\infty(e_\bullet)$ is shift-periodic.
Recall that, as a substitution of $\mathcal E_\bullet^*$, $\sigma$ is primitive.
\autoref{res: no big power in the orbit of the substitution} implies that there exists an integer $q \geq 2$ and a primitive word $u=e_1e_2\dots e_\ell$ of $\mathcal E_\bullet^*$ such that $\sigma^\infty(e_\bullet)=u^\infty$ and $\sigma(u)=u^q$.
Notice that $e_1=e_\bullet$.
This means, in particular, that the infinite path $f_\#^\infty(e_\bullet)$ is obtained as a concatenation 
\begin{equation*}
	f_\#^\infty(e_\bullet)=\gamma_0\cdot\gamma_1\cdot \gamma_2 \cdot\ldots\cdot\gamma_k\cdot\ldots,
\end{equation*}
of loops 
\begin{equation*}
	\gamma_k= e_1\cdot \alpha_{k\ell+1}\cdot e_2\cdot\alpha_{k\ell+2}\cdot\ldots\cdot  e_\ell\cdot\alpha_{(k+1)\ell},
\end{equation*}
where the $\alpha_i$ are (possibly trivial) yellow paths.
Moreover, if $\alpha_i$ is non trivial, then it is a maximal yellow subpath of the image by power of $f_\#$ of a red edge.
By assumption, none of them is a loop.
Recall that $x_0$ is the initial point of $e_1=e_\bullet$.
For every $i \in \intvald 1\ell$, $y_i$ and $x_i$ respectively stand for the initial and the terminal points of $\alpha_i$.
In particular, $x_0=x_\ell$.
We now focus on the path $\gamma=\gamma_0$:
\begin{equation*}
	\gamma=e_1\cdot \alpha_1\cdot e_2\cdot \alpha_2 \cdot\ldots\cdot  e_\ell\cdot \alpha_\ell.
\end{equation*}

\begin{lemm}
\label{res: f(gamma) prefix f-infty(e)}
	The path $f_\#(\gamma)$ is exactly $\gamma_0\dots \gamma_{q-1}$. 
	In particular, it is an initial subpath of the infinite path $f_\#^\infty(e_\bullet)$.
\end{lemm}

\begin{proof}
	By construction, there exists $p\in \N$ such that $\gamma_0\cdot\gamma_1$ is a proper initial subpath of $f_\#^p(e_\bullet)$.
	Moreover, the terminal point of $\gamma_0\gamma_1$, which is also the initial vertex of the red edge $e_1=e_\bullet$, is splitting point of the yellow-red splitting of $f_\#^p(e_\bullet)$.
	Thus $f_\#(\gamma_0\gamma_1) = f_\#(\gamma_0)  \cdot f_\#(\gamma_1)$ is an initial subpath of $f_\#^{p+1}(e_\bullet)$, hence of $f_\#^\infty(e_\bullet)$.
	However, by \autoref{res: substitution and RTT coincide on the red}
	\begin{equation*}
		\Red{f_\#(\gamma_0)} = \sigma(\Red{\gamma_0}) =  \sigma(u)  = u^q = \Red{ \gamma_0\dots \gamma_{q-1}}.
	\end{equation*}
	It follows that there exists a subpath $\alpha'$ of $\alpha_{q\ell}$ such that
	\begin{equation*}
		f_\#(\gamma_0) = \left[\gamma_0\dots \gamma_{q-2}\right]\left[e_1 \alpha_{(q-1)\ell+1} e_2\alpha_{(q-1)\ell+2}\ldots  e_\ell\alpha'\right].
	\end{equation*}
	On the other hand, $\gamma_1$ and thus $f_\#(\gamma_1)$ starts with the red edge $e_1=e_\bullet$.
	Since $f_\#(\gamma_0)  \cdot f_\#(\gamma_1)$ is an initial subpath of $f_\#^{p+1}(e_\bullet)$, the path $\alpha'$ is necessarily the whole $\alpha_{q\ell}$.
	Consequently $f_\#(\gamma) = \gamma_0\dots \gamma_{q-1}$. 
\end{proof}

\paragraph{The graph $\Gamma$ and the loop $\hat{\gamma}$.}

Let $i \in \intvald 1\ell$.
We define $\hat G_i$ to be a copy of the largest connected yellow subgraph of $G$ containing $y_i$.
We denote by $\hat \alpha_i$ (\resp $\hat y_i$ and $\hat x_i$) the path $\alpha_i$ (\resp the vertices $y_i$ and $x_i$) viewed as a path of $\hat G_i$ (\resp as vertices of $\hat G_i$).

\paragraph{}
We now construct a graph $\Gamma$ as follows.
We start with the disjoint union of the $\hat G_i$ for $i \in \intvald 1\ell$.
Then, for every $i \in \intvald 1\ell$, we add an oriented edge $\hat e_i$ whose initial and terminal points are respectively $\hat x_{i-1}$ and $\hat y_i$.
The reverse edge $\hat e_i^{-1}$ is attached accordingly.
In this process we think about the indices $i$ as elements of $\Z/\ell\Z$.
In particular, $\hat x_0$ should be understood as the point $\hat x_\ell$ of $\hat G_\ell$.
We denote by $\Gamma$ the graph obtained in this way (see \autoref{fig: graph Gamma}).
\begin{figure}[h]
	\centering
	\includegraphics[width=0.9\textwidth]{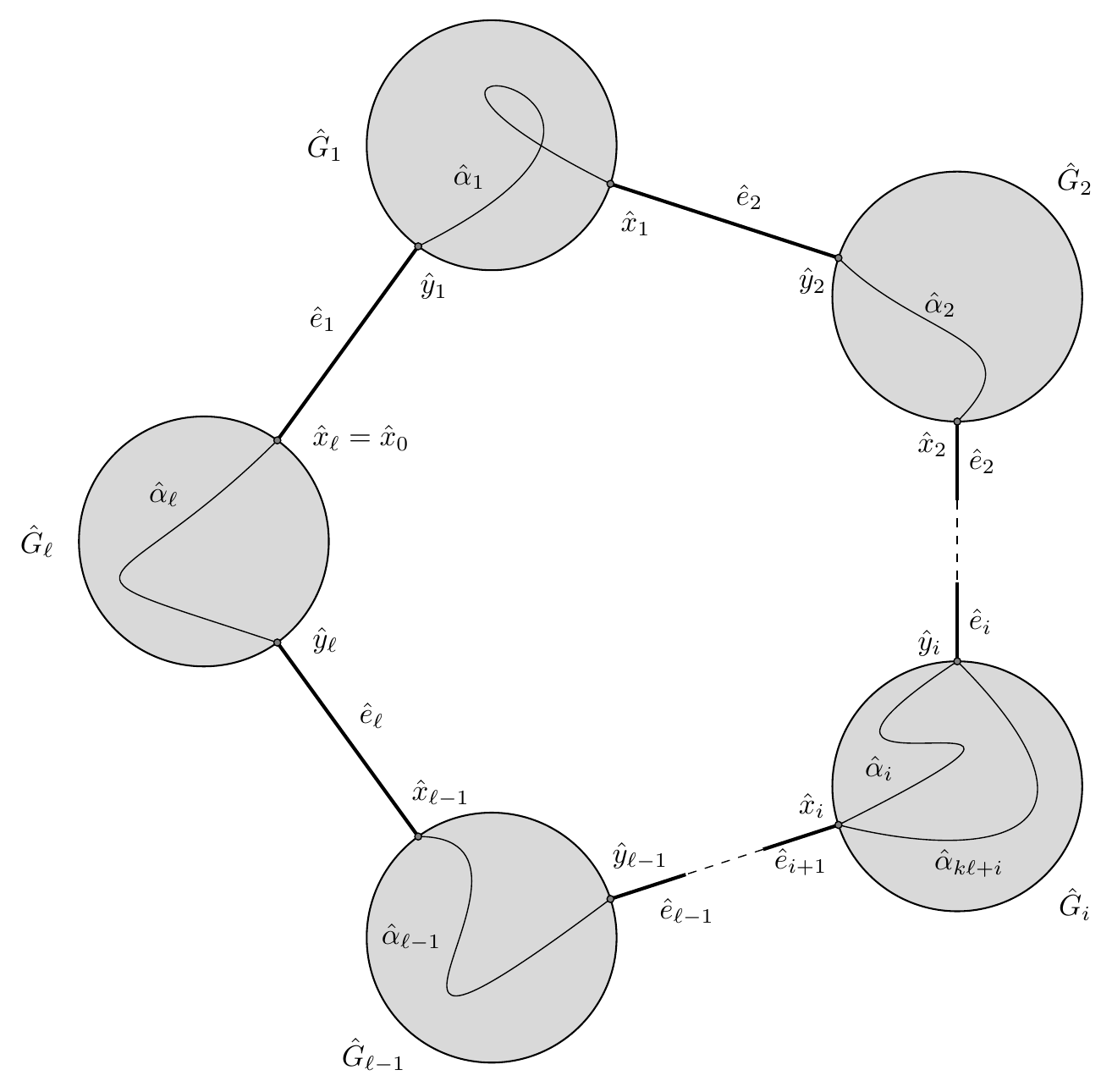}
	\caption{The graph $\Gamma$}
	\label{fig: graph Gamma}
\end{figure}%
Let $\rho$ be the graph morphism $\rho : \Gamma \rightarrow G$ such that for every $i \in \intvald 1\ell$, $\rho(\hat e_i) = e_i$ and the restriction of $\rho$ to $\hat G_i$ is the natural embedding $\hat G_i \hookrightarrow G$.
We color the edges of $\Gamma$ by the color of their images under $\rho$.
In other words, the edges $\hat e_i$ are red, whereas the edges of the subgraphs $\hat G_i$ are yellow.
By construction, the loop $\hat \gamma$ defined below is a lift of $\gamma$ in $\Gamma$.
\begin{equation*}
	\hat \gamma=\hat e_1 \hat \alpha_1 \hat e_2 \hat \alpha_2 \dots  \hat e_\ell \hat \alpha_\ell.
\end{equation*}

\paragraph{The subgroup $\mathbf H$.}
We denote by $\mathbf H$ the fundamental group $\pi_1(\Gamma,\hat x_0)$.
Let us choose a maximal tree $T_i$ in each $\hat G_i$.
The union $T$ defined below is a maximal tree of $\Gamma$. 
\begin{equation*}
	T = \left( \bigcup_{i=1}^\ell T_i\right) \cup\left(\bigcup_{i=1}^{\ell-1}\hat e_i\right).
\end{equation*}
For every edge $e$ of $\Gamma$ not in $T$, we write $\beta_e$ for the path contained in $T$ starting at $\hat x_0$ and ending at the initial vertex of $e$.
We define $h_e$ as the element of $\mathbf H$ represented by  $\beta_e e\beta_{e^{-1}}^{-1}$.
Let $i \in \intvald 1\ell$.
For each unoriented edge of $\hat G_i \setminus T_i$, we chose one of the two corresponding oriented edges.
We denote then by $\prefF_i$ the preferred set of oriented edges obtained in this way. 
We write $\prefF$ for the union 
\begin{equation*}
	\prefF= \bigcup_{i=1}^\ell \prefF_i.
\end{equation*}

\begin{lemm}
\label{res: gamma hat primitive}
	Let $h$ be the element of $\mathbf H$ represented by $\hat \gamma$.
	The family $\mathcal B$ obtained by taking the union of $(h_e)_{e \in \prefF}$ and $\{h\}$ is a free basis of $\mathbf H$.
\end{lemm}

\begin{proof}
	It follows from the definition of $\prefF$ that the family $(h_e)_{e \in \prefF \cup \{\hat e_\ell\}}$ is a free basis of $\mathbf H$.
	By construction of $\Gamma$,  $h = g \cdot h_{\hat e_\ell}$ where $g$ is a product of some $h_e$ with $e \in \prefF \cup\prefF^{-1}$.
	This implies that the family $(h_e)_{e \in \prefF}$ together with $h$ forms a free basis of $\mathbf H$.
\end{proof}

	Let $k \in \N$ and $\in i \in \intvald 1\ell$.
	The path $\alpha_{k\ell+i}$ and $\alpha_i$ have the same endpoints, namely $y_i$  (the terminal point of $e_i$) and $x_i$ (the initial point of $e_{i+1}$).
	In particular, they are contained in the same maximal yellow connected component of $G$.
	We denote by $\hat \alpha_{k\ell+i}$ the copy in $\hat G_i$ of $\alpha_{k\ell+i}$ (See \autoref{fig: graph Gamma}).
	We put 
	\begin{equation*}
		\hat \gamma_k= \hat e_1 \hat \alpha_{k\ell+1}\hat e_2\hat \alpha_{k\ell+2}\dots \hat e_\ell\hat \alpha_{(k+1)\ell}.
	\end{equation*}
	By construction $\hat \gamma_k$ is a loop of $\Gamma$ based at $\hat x_0$ lifting $\gamma_k$ (i.e. $\rho \circ \hat \gamma_k = \gamma_k$).
	
\begin{lemm}
\label{res: hat gamma k in terms of hat gamma}
	Let $h$ be the element of $\mathbf H$ represented by $\hat \gamma$.
	Let $k \in \N$. 
	There exists $g$ in the normal subgroup generated by $(h_e)_{e \in \prefF}$ such that the element of $\mathbf H$ represented by the loop $\hat \gamma_k$ is $gh$.
\end{lemm}

\begin{proof}
	It just follows from the following equality
	\begin{align*}
		\hat \gamma_k  = &
		\left[ \hat e_1\left(\hat \alpha_{k\ell+1}\hat \alpha^{-1}_1\right)\hat e^{-1}_1\right]  
		\left[ \hat e_1 \hat \alpha_1\hat e_2\left(\hat \alpha_{k\ell+2}\hat \alpha^{-1}_2\right)\hat e^{-1}_2\hat \alpha^{-1}_1\hat e^{-1}_1\right] \dots \\
		&\left[ \hat e_1 \hat \alpha_1\hat e_2\dots \hat e_{\ell-1}\left(\hat \alpha_{(k+1)\ell-1}\hat \alpha^{-1}_{\ell-1}\right) \hat e^{-1}_{\ell-1} \dots \hat e^{-1}_2\hat \alpha^{-1}_1\hat e^{-1}_1\right]  \\
		&\left[\hat \gamma \left( \hat \alpha_\ell^{-1}\hat \alpha_{(k+1)\ell}\right) \hat \gamma^{-1}\right]\hat \gamma\qedhere
	\end{align*}
\end{proof}

\begin{lemm}
\label{res: p loc injective}
	The map $\rho : \Gamma \rightarrow G$ is locally injective.
\end{lemm}

\begin{proof}
	We proof this lemma by contradiction.
	Let $\hat e$ and $\hat e'$ be two distinct edges of $\Gamma$ with the same initial vertex $\hat v$.
	Suppose that $\rho(\hat e) = \rho(\hat e')$.
	There exists $i \in \intvald 1\ell$ such that $\hat v$ is a vertex of $\hat G_i$.
	By construction $\rho$ preserves the color of the edges, thus $\hat e$ and $\hat e'$ necessarily have the same color.
	We distinguish two cases.
	Assume first that $\hat e$ and $\hat e'$ are both yellow edges.
	Recall that the restriction of $\rho$ to $\hat G_i$ is the inclusion $\hat G_i \hookrightarrow G$. 
	Thus $\hat e = \hat e'$.
	Contradiction.
	Assume now that $\hat e$ and $\hat e'$ are red.
	By construction of $\Gamma$, at most two red edges have an initial vertex in $\hat G_i$.
	Without loss of generality we can assume that $\hat e^{-1} = \hat e_i$ and $\hat e' = \hat e_{i+1}$ (as previously, if $i=\ell$ then $\hat e_{i+1}$ correspond to $\hat e_1$).  
	Then $\rho(\hat v)$ is the terminal  vertex $y_i$ of $e_i$ and the initial vertex $x_i$ of $e_{i+1}$.
	Thus the yellow path $\alpha_i$ is either trivial or a loop of $G$.
	By assumption, it cannot be a loop, thus $\alpha_i$ is trivial and $e_{i+1} = \rho(\hat e') = \rho(\hat e) = e_i^{-1}$.
	It contradicts the fact that $\gamma$ is a path.
	Consequently $\rho$ is locally injective.
\end{proof}

If follows from the lemma that $\rho$ induces an embedding $\rho_*$ from $\mathbf H$ into $\pi_1(G,x_0)$.
From now on, we identify $\mathbf H$ with its image in $\pi_1(G,x_0)$.

\paragraph{The automorphism induced on $\mathbf{H}$.}
Recall that $\phi$ is the automorphism of $\pi_1(G, x_0)$ in the outer class $\Phi$ induced by $f$.
We now prove that $\phi$ induces an automorphism of $\mathbf H$.
To that end we lift the RTT $f$ into a map $\hat f : \Gamma \rightarrow \Gamma$.

\begin{lemm}
\label{res: lifting f in Gamma}
	There exists a continuous map $\hat f : \Gamma \rightarrow \Gamma$ satisfying the followings.
	\begin{enumerate}
		\item $f \circ \rho = \rho \circ \hat f$
		\item $\hat f(\hat \gamma)$ is homotopic relatively to its endpoints to $\hat \gamma_0 \dots \hat \gamma_{q-1}$.
	\end{enumerate}
\end{lemm}

\begin{proof}
	For every $i \in \intvald 1\ell$, we denote by $\Gamma_i$ the subgraph of $\Gamma$ containing the red edges $\hat e_1, \dots, \hat e_i$ and the yellow graphs $\hat G_1, \dots, \hat G_i$.
	In particular, $\Gamma_\ell$ is the whole graph $\Gamma$.
	By convention, we put $\Gamma_0 = \{ \hat x_0 \}$.	
	Let $i \in \intvald 1\ell$.
	The path $\gamma$ can be split as follows
	\begin{equation*}
		\gamma = (e_1\alpha_1\dots e_i \alpha_i)\cdot (e_{i+1}\alpha_{i+1} \dots e_\ell \alpha_\ell).
	\end{equation*}
	Therefore $f_\#(e_1\alpha_1\dots e_i \alpha_i)$ is an initial subpath of $f_\#(\gamma) = \gamma_0 \dots \gamma_{q-1}$ (see \autoref{res: f(gamma) prefix f-infty(e)}).
	However the path $\hat \gamma_0 \dots \hat \gamma_{q-1}$ is by construction a lift in $\Gamma$ of $\gamma_0 \dots \gamma_{q-1}$.
	Thus there exists a path $\hat \beta_i$ in $\Gamma$ starting at $\hat x_0$ such that $\rho \circ \hat \beta_i = f_\#(e_1\alpha_1\dots e_i \alpha_i)$.
	In particular $\hat \beta_\ell = \hat \gamma_0 \dots \hat \gamma_{q-1}$.
	In addition, we define $\hat \beta_0$ to be the trivial path equal to $\hat x_0$.
	
	\paragraph{}
	We claim that for every $i \in \intvald 0\ell$, there exists a continuous map $\hat f_i : \Gamma_i \rightarrow \Gamma$ satisfying the following
	\begin{enumerate}
		\item $f \circ \rho = \rho \circ \hat f_i$,
		\item $\hat f_i(\hat e_1\hat \alpha_1\dots \hat e_i \hat \alpha_i)$ is homotopic relatively to its endpoints to $\hat \beta_i$.
	\end{enumerate}
	We proceed by induction on $i$.
	By assumption $f$ fixes the vertex $x_0$. 
	We put $\hat f_0 (\hat x_0) = \hat x_0$, hence the claim holds for $i=0$.
	Assume now that the claim holds for $i \in \intvald 0{\ell-1}$.
	Our goal is to extend $\hat f_i$ into a map $\hat f_{i+1} : \Gamma_{i+1} \rightarrow \Gamma$.
	We start with the following observation.
	By construction the path $\gamma$ splits as follows
	\begin{equation*}
		\gamma = (e_1\alpha_1 \dots e_i \alpha_i)\cdot e_{i+1} \cdot \alpha_{i+1} \cdot (e_{i+2}\alpha_{i+2}\dots e_\ell\alpha_\ell).
	\end{equation*}
	Therefore we have 
	\begin{equation*}
		f_\#(\gamma) = f_\#(e_1\alpha_1 \dots e_i \alpha_i)\cdot f(e_{i+1}) \cdot f_\#(\alpha_{i+1}) \cdot f_\#(e_{i+2}\alpha_{i+2}\dots e_\ell\alpha_\ell).
	\end{equation*}
	In particular $f(e_{i+1})$ is a subpath of $f_\#(\gamma)$.
	The path $\hat e_1 \hat \alpha_1 \dots \hat e_i \hat \alpha_i$ is contained in $\Gamma_i$.
	According to the induction hypothesis, $\hat f_i(\hat e_1 \hat \alpha_1 \dots \hat e_i \hat \alpha_i)$ is a lift of $f(e_1\alpha_1 \dots e_i \alpha_i)$ which is homotopic relatively to its endpoints to $\hat \beta_i$.
	In particular $\hat f_i(\hat x_i$) is the endpoint of $\hat \beta_i$ and a preimage in $\Gamma$ of $f(x_i)$.
	Recall that $\hat \beta_i$ is the lift in $\Gamma$ of $f_\#(e_1\alpha_1 \dots e_i \alpha_i)$.
	Therefore we can define $\hat f_{i+1}(\hat e_{i+1})$ to be the path of $\Gamma$ starting at $\hat f_i(\hat x_i)$ that lifts $f(e_{i+1})$.
	Let us now focus on $\hat G_{i+1}$.
	Since $e_{i+1}$ is a red edge, its image under $f$ starts and ends by a red edge.
	In particular, there exists $j \in \intvald 1\ell$ such that the last edge of $f(e_{i+1})$ is $e_j$. 
	On the other hand $f$ is continuous and sends yellow edges to yellow paths.
	Therefore it maps the largest yellow connected component of $G$ containing $y_{i+1}$ to the largest yellow connected component of $G$ containing the endpoint of $f(e_{i+1})$, i.e. $y_j$.
	It provides a map from $\hat G_{i+1}$ to $\hat G_j$, which completes the definition of $\hat f_{i+1}$.
		
	\paragraph{}
	It remains to prove that $\hat f_{i+1}(\hat e_1\hat \alpha_1\dots \hat e_{i+1} \hat \alpha_{i+1})$ is homotopic relatively to its endpoints to $\hat \beta_{i+1}$.
	By construction 
	\begin{equation*}
		\hat f_{i+1}(\hat e_1\hat \alpha_1\dots \hat e_{i+1} \hat \alpha_{i+1}) 
		= \hat f_i(\hat e_1 \hat \alpha_1 \dots \hat e_i \hat \alpha_i)\hat f_{i+1}(\hat e_1)\hat f_{i+1}(\hat \alpha_{i+1}) 
	\end{equation*}
	In particular it is homotopic relatively to its endpoints to $\hat \beta_i\hat f_{i+1}(\hat e_1)\hat f_{i+1}(\hat \alpha_{i+1})$.
	However, $\hat \beta_i \hat f(\hat e_i)$ is the lift starting at $\hat x_0$ of $f_\#(e_1\alpha_1 \dots e_i \alpha_i e_{i+1})$.
	Thus it is sufficient to prove that $\hat f_{i+1}(\hat \alpha_{i+1})$ is homotopic relatively to its endpoints to the lift starting at $\hat f_{i+1}(\hat y_{i+1})$ of $f_\#(\alpha_{i+1})$.
	However, these last paths all belong to $\hat G_j$ and the restriction of $\rho$ to $\hat G_j$ is the natural embedding $\hat G_j \hookrightarrow G$.
	The conclusion follows then from the fact that $f(\alpha_{i+1})$ and $f_\#(\alpha_{i+1})$ are homotopic relatively to their endpoints in the yellow connected component of $G$ they belong to.
	Thus the claim holds for $i+1$
	
	\paragraph{}
	Note that if $i=\ell$ our construction provides two definitions for $\hat f_\ell(\hat x_0)$.
	On the one hand $\hat f_0(\hat x_0) = \hat x_0$.
	On the other hand $\hat x_0 = \hat x_\ell$ is the endpoint of $\hat \alpha_\ell$ that belongs to $\hat G_l$.
	However $\hat f_\ell(\hat \gamma)$ and $\hat \gamma_0 \dots \hat \gamma_{q-1}$ are homotopic relatively to their endpoints.
	Thus the endpoints of $\hat \gamma$ (which is $\hat x_0$) is sent to the endpoint of $\hat \gamma_{q-1}$, i.e. $\hat x_0$.
	Hence $\hat f_\ell$ is well defined.
	We conclude the proof by taking $\hat f = \hat f_\ell$.	
\end{proof}

\begin{lemm}
\label{res: subgroup H invariant under power of phi}
	The map $\phi$ induces an automorphism of $\mathbf H$.
\end{lemm}

\begin{proof}
	It follows from \autoref{res: lifting f in Gamma} that $\phi(\mathbf H)$ is a subgroup of $\mathbf H$.
	As a consequence of the LERF property, it implies that the restriction of $\phi$ to $\mathbf H$ is an automorphism -- see for instance Lemma~6.0.6 in \cite{BesFeiHan00}.
\end{proof}

\paragraph{The abelianization of $\mathbf H$.}

	We now completes the proof of \autoref{res: sigma(e) not shift periodic}.
	Let $d$ be the rank of the free group $\mathbf H$.
	We consider the the abelianization morphism $\mathbf H\rightarrow \Z^d$.
	In particular $\phi$ induces an automorphism $\phi_{ab}$ of $\Z^d$.
	We denote by $\mathcal C$ the image in $\Z^d$ of the free basis $\mathcal B$ of $\mathbf H$ given by \autoref{res: gamma hat primitive}.
	The first $(d-1)$ elements of $\mathcal B$ (the ones corresponding to oriented edges in $\prefF$) are conjugates of a yellow loops of $\Gamma$. 
	However $f$ and thus $\hat f$ maps yellow edges to yellow edges.
	Hence the subgroup $\Z^{d-1}$ generated by the first $(d-1)$ elements of $\mathcal C$ is invariant under $\phi_{ab}$.
	By \autoref{res: lifting f in Gamma}, $\hat f (\hat \gamma)$ is homotopic relatively to $\{\hat x_0\}$ to $\hat \gamma_0 \hat \gamma_1 \cdots \hat \gamma_{q-1}$.
	It follows from \autoref{res: hat gamma k in terms of hat gamma} that the matrix $R$ of $\phi_{ab}$ in the basis $\mathcal C$ has the following shape:
	\begin{equation*}
		R=\left(
		\begin{array}{ccc|c}
		\star & \dots & \star & \star \\
		\vdots & & \vdots & \vdots \\
		\star & \dots & \star & \star \\
		\hline
		0 & \dots & 0 & q
		\end{array}
		\right).
	\end{equation*}
	Since $q\geq 2$, the determinant of $R$ cannot be invertible in $\Z$, which contradicts the fact that  $\phi_{ab}$ is an automorphism.
	We have thus proved that $\sigma^\infty(e_\bullet)$ is not shift-periodic.
\end{proof}

\begin{prop}
\label{res: no power in the red}
	There exists an integer $m$ such that for every $p \in \N$, as a word over $\mathcal E_\bullet$, $\Red{f^p_\#(e_\bullet)}$ does not contains an $m$-th power.
\end{prop}

\begin{proof}
	According to \autoref{res: substitution and RTT coincide on the red}, for every $p \in \N$, the red word associated to $f^p_\#(e_\bullet)$ is exactly $\sigma^p(e_\bullet)$.
	However the substitution $\sigma$ is primitive and the infinite word $\sigma^\infty(e_\bullet)$ is not shift-periodic.
	Hence the result follows from \autoref{res: no big power in the orbit of the substitution}.
\end{proof}

\section{The automorphism of $\burn rn$ induced by $\phi$.}
\label{sec:proof main}

\subsection{A criterium of non triviality in $\burn{r}{n}$}
\label{sec:criterium}

\paragraph{}
Let us have a pause in order to introduce a key tool for the sequel of the proof of \autoref{res: reduction - one exponential stratum}.

\paragraph{}
Let $(X,x_0)$ be a pointed simplicial tree.
Given two points $x$ and $x'$ of $X$, $\dist x{x'}$ denotes the distance between them whereas $\geo x{x'}$ stands for the geodesic joining $x$ and $x'$.
Let $g$ be an isometry of $X$.
Its \emph{translation length}, denoted by $\len g$, is the quantity $\len g = \inf_{x\in X} \dist {gx}x$.
If $X$ is the Cayley graph of $\free r$, then $\len g$ is exactly the length of the conjugacy class of $g \in \free r$.
The set of points $A_g = \left\{ x \in X \mid \dist {gx}x = \len g\right\}$ is called the \emph{axis} of $g$.
It is a subtree of $X$.
It is known that either $\len g = 0$ and $A_g$ is the set of fixed points of $g$ or $\len g > 0$ and $A_g$ is a bi-infinite geodesic on which $g$ acts by translation of length $\len g$.
In the first case $g$ is said to be \emph{elliptic}, in the second one \emph{hyperbolic}.
For  more details we refer the reader to \cite{CulMor87}.

\paragraph{}Let $n \in \N$.
We explain now a criterium to decide whether two elements of $\free r$ have the same image in $\burn rn$ or not.
To that end we assume that $\free r$ acts by isometries on $X$.

\begin{defi}
	Let $n \in \N$ and $\xi \in \R_+$.
	Let $y$ and $z$ be two points of $X$.
	We say that $z$ is the image of $y$ by an \emph{$(n,\xi)$-elementary move} (or simply \emph{elementary move}) if there is a hyperbolic element $u \in \free r$ such that
	\begin{enumerate}
		\item $\diam \left( \geo {x_0}y \cap A_u \right) \geq (n/2 -\xi )\len u$
		\item $z = u^{-n}y$.	
	\end{enumerate}
	The point $z$ is the image of $y$ by a \emph{sequence of $(n,\xi)$-elementary moves} if there is a finite sequence $y=y_0, y_1, \dots ,y_\ell=z$ such that for all $i \in \intvald 0{\ell-1}$, 
	$y_{i+1}$ is the image of $y_i$ by an $(n,\xi)$-elementary move.
\end{defi}

\begin{remark} 
	Roughly speaking an elementary move allows us to replace a subword of the form $v^m$ by $v^{m-n}$ provided $m$ is sufficiently large.
	Assume that $X$ is the Cayley graph of $\free r$ and $x_0$ the vertex of $X$ that corresponds to 1.
	Let $g$ be an element of $\free r$.
	The reduced word $w$ which represents $g$ labels the geodesic between $x_0$ and $gx_0$.
	Let us suppose now that $w$ can be written (as a reduced word) $w=pv^ms$ with $m \geq n/2-\xi$.
	It follows that 
	\begin{displaymath}
		\diam \left( \geo {x_0}{gx_0} \cap A_u \right) \geq \len {u^m} \geq (n/2-\xi) \len u,
	\end{displaymath}
	where $u$ is the element of $\free r$ represented by $pvp^{-1}$.
	Thus $u^{-n}g$ which is represented by $pv^{m-n}s$ is the image of $g$ by an elementary move.
	Later in the proof, the tree $X$ will be the universal cover of an RTT.
	Therefore this formulation, which extends the idea of substituting subwords, is more appropriate for our purpose.
\end{remark}

\begin{prop}[Coulon {\cite[Theorem 4.12]{Coulon:2012tj}}]
\label{res: criterium}
	Assume that $\free r$ acts properly co-compactly by isometries on $(X,x_0)$.
	There exist $n_1 \in \N$ and $\xi \in \R_+$ such that for every odd exponent $n \geq n_1$ the following holds.
	Two isometries $g,g' \in \free r$ have the same image in $\burn rn$ if and only if there exist two finite sequences of $(n,\xi)$-elementary moves which respectively send $gx_0$ and $g'x_0$ to the same point of $X$.
\end{prop}

This proposition provides in particular a criterion for detecting trivial elements in $\burn rn$.

\begin{coro}
\label{res: criterium - coro}
	Assume that $\free r$ acts properly co-compactly by isometries on $(X,x_0)$.
	There exist $n_1 \in \N$ and $\xi \in \R_+$ such that for every odd exponent $n \geq n_1$ the following holds.
	An element $g \in \free r$ is trivial in $\burn rn$ if and only if there exists a finite sequence of $(n,\xi)$-elementary moves which sends $gx_0$ to $x_0$.
\end{coro}

However, despite the similarity with the word problem in a group, \autoref{res: criterium - coro} is not equivalent to \autoref{res: criterium}.
This comes from the fact that $(n,\xi)$-elementary moves are not symmetric.
One first has to see a large power along the geodesic $\geo {x_0}{gx_0}$ before performing an elementary move.
For instance, if $a$ and $b$ are two distinct primitive elements of $\free r$, there is no sequence of elementary moves that sends $a^n$ to $b^n$.
\autoref{res: criterium - coro} only implies a weaker form of \autoref{res: criterium} in the sense that we need to allow a larger class of elementary moves: those of the form $pv^ms \rightarrow pv^{m-n}s$ with $m \geq n/4 -\xi/2$.

\paragraph{}
In our situation, we will apply this criterion with two elements of the form $g$ and $\phi^p(g)$, where $g \in \free r$ and $\phi$ is the automorphism we want to study.
The theory of train-track provides many information about the path $\geo{x_0}{\phi^p(g)x_0}$.
Therefore it is also more natural to have a criterion that uses conditions on $\geo{x_0}{gx_0}$ and $\geo{x_0}{\phi^p(g)x_0}$ rather than $\geo{gx_0}{\phi^p(g)x_0}$.

\paragraph{}
\autoref{res: criterium} is ``well-known'' from the specialists of Burnside's groups. 
One can find in Adian's \cite{Adi79} or Ol'shanski\u\i's work \cite{Olc82} statements close to \autoref{res: criterium}.
However, to our knowledge it has never been stated in this simple way.
Let us explain the parallel with Ol'shanski\u\i's solution to the Burnside problem.
His proof is written in terms of van Kampen diagrams.
In particular, Lemma~5.5 of \cite{Olc82} says the following.
Let $w$ be a word over a basis of $\free r$.
For a sufficiently large exponent $n$ if $w$ induces a trivial element of $\burn rn$, then there exists a van Kampen diagram whose boundary is labelled by $w$.
Moreover this diagram contains a $2$-cell such that at least 1/3 of its boundary is a subpath of $w$.
Such a cell represents a relation of the form $v^n$.
Removing this $2$-cell corresponds to performing an elementary move.
A proof by induction on the number of $2$-cells would give a statement similar to \autoref{res: criterium - coro}.
One needs though to allow weaker elementary moves of the form $pv^ms \rightarrow pv^{m-n}s$ with $m \geq n/3$.
However a sharpening of Lemma~5.5 of \cite{Olc82} should provide the same estimates.

\paragraph{}
An other difference between the two approaches is the following.
A.Y.~Ol'shan\-ski\u\i\ works with words over a basis of $\free r$: this can be interpreted in terms of an action of $\free r$ on its Cayley graph.
In our context we have to  consider the action of $\free r$ on a tree $X$ which is no more its Cayley graph, but the universal cover of a train-track.
Therefore one could not apply directly Ol'shanksi\u\i's result. 
One would need to use before an equivariant quasi-isometry to transpose his statement in an arbitrary tree.

\subsection{Performing elementary moves in $\tilde G$}
\label{sec: tracking powers}

\paragraph{}
We get back to the proof of \autoref{res: reduction - one exponential stratum}.
The notations are the same as the ones of \autoref{sec:no big powers}.

\paragraph{Metrics on $\tilde{G}$.}

For our purpose, the pointed tree $(X,x_0)$ that appears in \autoref{res: criterium} will be the universal cover $(\tilde G, \tilde x_0)$ of $G$ where $\tilde x_0$ is preimage of $x_0$.
By declaring that any edge of $\tilde{G}$ is isometric to the unit real segment $[0,1]$ we obtain an $\free r$-invariant length metric on $\tilde{G}$: the \emph{combinatorial metric}. 
We denote by $|\alpha|$ the resulting \emph{combinatorial length} of a path $\alpha$ in $\tilde{G}$.

\paragraph{}
We also define a pseudo-length-metric on $\tilde{G}$ in the following way.
We first consider that any yellow edge has length zero.
Recall that $\mathcal E$ is the set of all the oriented red edges of $G$. 
We chose a preferred set of oriented edges $\vec{\mathcal E}$.
Recall that the transition matrix $M$ of the red stratum of $f$ is aperiodic.
We denote by $\lambda > 1$ the Perron-Frobenius dominant eigenvalue of $M$ and we consider a positive right eigenvector $l=(l_e)_{e\in \vec{\mathcal E}}$ associated to $\lambda$.
We declare the lifts of $e$ isometric to the real segment $[0,l_e]$.
The resulting pseudo-metric is called the \emph{PF-pseudo-metric}.
We denote by $|\alpha|_\text{PF}$ the resulting length of the path $\alpha$ in $\tilde{G}$: $|\alpha|_\text{PF}$ is called the \emph{PF-length} of $\alpha$.
This length only depends on the red word $\Red \alpha \in \mathcal E^*$.
If $\alpha$ is a red-legal path, we thus get that for all $p\in\N$, 
\begin{equation*}
	|f^p_\#(\alpha)|_\text{PF}=\lambda^p|\alpha|_\text{PF}.
\end{equation*}
Unless stated otherwise we will work with $\tilde G$ endowed with the combinatorial metric.

\paragraph{The element $g$ and its orbit.}
Recall that $e_\bullet$ is the red edge fixed at the beginning of \autoref{sec:no big powers}.
For all $p\in\N$, $f^p_\#(e_\bullet)$ is a path starting by $e_\bullet$.
Its yellow-red decomposition is a splitting.
The red stratum $H$ is aperiodic.
Thus if $p$ is a sufficiently large integer, one can find an other occurrence of $e_\bullet$ in $f^p_\#(e_\bullet)$:
$f^p_\#(e_\bullet)= e_\bullet \nu_0e_\bullet \nu_1$.
The path $\nu=e\nu_0$ is a red-legal circuit, and the yellow-red decomposition of $\nu$ is a splitting. 
We denote by $g$ the element of $\pi_1(G,x_0)$ represented by $\nu$.
By construction the geodesic $\geo {\tilde x_0}{g\tilde x_0}$ is the lift in $\tilde G$ of $\nu$ starting at $\tilde x_0$.

\begin{lemm}
\label{res; no power in the red - geometric statement}
	There exists an integer $n_2$ with the following property.
	Let $p \in \N$.
	Let $\beta$ be a path of $\tilde G$ such that the red words respectively associated to $\beta$ and $f^p_\#(\nu)$ are the same.
	For all $u \in \free r \setminus \{1\}$, if
	\begin{displaymath}
		\diam \left(\beta\cap A_u \right) > n_2 \len u,
	\end{displaymath}
	then the axis of $u$ only contains yellow edges.
\end{lemm}

\begin{proof}
	By construction there exists $p_0 \in \N$ such that $\nu$ is a prefix of $f_\#^{p_0}(e_\bullet)$.
	More generally, for every $p \in \N$, $f_\#^p(\nu)$ is a prefix of $f_\#^{p+p_0}(e_\bullet)$.
	According to \autoref{res: no power in the red}, there exists $m \in \N$ such that for every $p \in \N$, the red word associated to $f_\#^p(\nu)$ does not contain an $m$-th power.
	Put $n_2 = m +2$.
	Let $\beta$ be a path of $\tilde G$ and $p$ an integer such that the red words respectively associated to $\beta$ and $f^p_\#(\nu)$ are the same.
	Let $u$ be a non-trivial element of $\free r$ such that 
	\begin{displaymath}
		\diam \left(\beta\cap A_u \right) > n_2 \len u.
	\end{displaymath}
	In particular there is a vertex $x \in A_u$ such that for every $j \in \intvald 0m$, $u^jx$ belongs to $\beta$.
	Assume now that $A_u$ contains a red edge $e$.
	Since $A_u$ is $u$-invariant bi-infinite geodesic, the geodesic $\geo x{ux}$ contains some red edges.
	In particular, if $\alpha$ stands for the path $\geo x{u^mx}$ then $\Red{\alpha}$ contains an $m$-th power.
	However, $\beta$ is a path of $\tilde G$.
	Consequently, $\geo x{u^mx}$ is a subset, hence a subpath of $\beta$.
	Therefore the red word associated to $\beta$ and thus to $f_\#^p(\nu)$ contains an $m$-th power.
	Contradiction.
\end{proof}

We finish this section by the proof of \autoref{res: reduction - one exponential stratum}.

\begin{proof}[Proof of \autoref{res: reduction - one exponential stratum}]
Our goal is to prove that for sufficiently large odd integers $n$ the sequence $(\phi^p(g))_{p\in\N}$ of elements of $\free r$ is embedded in $\burn rn$.
Since $\phi$ is an automorphism, it is sufficient to check that for all $p\in\N^*$, $\phi^p(g)\not\equiv g$ in $\burn rn$.
We are going to use the criterium of \autoref{sec:criterium}.
Recall that the geodesic $[\tilde x_0,g\tilde x_0]$ is a lift in $\tilde G$ of $\nu$.
We denote by $n_1$, $\xi$ and $n_2$ the constants given, accordingly, by \autoref{res: criterium} and \autoref{res; no power in the red - geometric statement}.
For the rest of the proof we fix an odd integer $n$ larger than 
\begin{equation*}
	n_0 = \max\left\{n_1,2n_2+2\xi+1, \dist{\tilde x_0}{g\tilde x_0}  \right\}.
\end{equation*}
Note that this lower bound only depends on the outer automorphism $\Phi$ and the RTT $f$.

\paragraph{}
Let $p \in \N^*$.
By construction the path $\beta=[\tilde x_0,\phi^p(g)\tilde x_0]$ is a lift of $f^p_\#(\nu)$.
Assume now that $\phi^p(g)\equiv g$ in $\burn rn$.
By \autoref{res: criterium}, there exists two sequences of $(n,\xi)$-elementary moves which respectively send $g$ and $\phi^p(g)$ to the same element of $\free r$.
However, we fixed $n \geq \dist{\tilde x_0}{g\tilde x_0}$.
Therefore no $(n,\xi)$-elementary move can be performed on $\geo{\tilde x_0}{g\tilde x_0}$.
It follows that there exists a sequence of $(n,\xi)$-elementary moves which sends $\phi^p(g)$ to $g$.
We denote by $\beta_i$ the reduced path obtained from $\beta$ after the $i$-th $(n,\xi)$-elementary move.
In particular $\beta_0 = \beta$.
We are going to show, by induction on $i$, that:
\begin{enumerate}[(H1)]
\item \label{it:yellow} for all $j\in\{0,\dots,i-1\}$, no maximal yellow subpath of $\beta_j$ is completely removed when operating the $(j+1)$-th elementary move,
\item \label{it:red} $\Red{\beta_{i}}=\Red{\beta}$.
\end{enumerate}
If $i = 0$, (H\ref{it:yellow}) is vacuous whereas (H\ref{it:red}) is obvious.
Assume that these two conditions hold for $i$.
We focus on the $(i+1)$-th elementary move.
Let us denote by $A_u$ the axis of  the elementary move performed on $\beta_i$.
In particular $\diam\left(\beta_i \cap A_u \right) \geq (n/2-\xi) \len u$ in $\tilde G$.
By hypothesis (H\ref{it:red}) the red words associated to $\beta_i$, $\beta$ and $f_\#^p(\nu)$ are the same.
By \autoref{res; no power in the red - geometric statement} the axis $A_u$ only contains yellow edges.

\paragraph{}
We consider now the $k$-th maximal yellow subpath $\alpha$ of $\beta_i$.
Assume that $\alpha$ is completely removed when performing the $(i+1)$-th elementary move.
Since $A_u$ is yellow, the image of $\alpha$ by the projection $\tilde G \rightarrow G$ must be a loop.
\begin{figure}[h]
	\centering
	\includegraphics[scale=0.8]{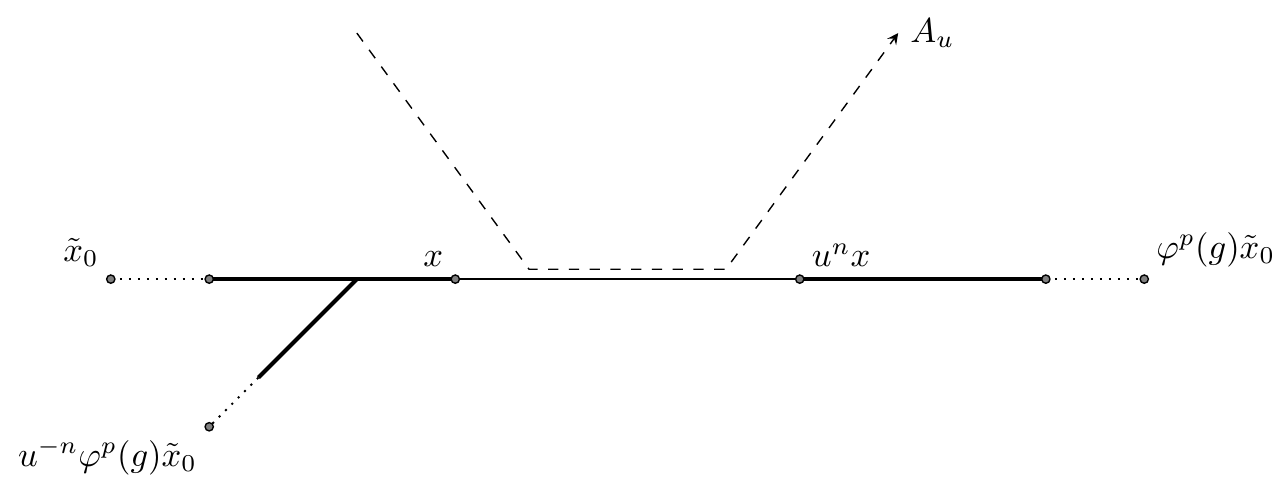}
	\caption{An elementary move removing a yellow path}
	\label{fig:collapsing-proof}
\end{figure}
Let $\alpha'$ be the $k$-th maximal yellow subpath of $\beta$.
By hypothesis (H\ref{it:yellow}), no maximal yellow subpath of $\beta_j$ is completely removed when operating the $(j+1)$-th elementary move for $j\in\{0,\dots,i-1\}$.
Thus $\alpha$ is obtained from $\alpha'$ by a sequence of elementary moves and the image of $\alpha'$ by $\tilde G \rightarrow G$ is also a loop.
On the other hand $f_\#^p(\nu)$ is a prefix of some path in the orbit under $f_\#$ of $e_\bullet$.
It follows from our original assumption that the image of $\alpha'$ in $G$ cannot be a loop.
Contradiction.
Thus (H\ref{it:yellow}) is true at step $i+1$.
Hypothesis (H\ref{it:yellow}) implies that $\Red{\beta_{i+1}}=\Red{\beta_i}$, and thus $\Red{\beta_{i+1}}=\Red{\beta}$ by hypothesis (H\ref{it:red}).
This gives (H\ref{it:red}) at step $i+1$, which completes the induction.

\paragraph{}
Recall that $(\beta_i)$ is the collection of paths obtained by the sequence of elementary moves which sends $\phi^p(g)$ to $g$.
It follows from the previous discussion that at each step $i$, $|\beta_i|_\text{PF} = |\beta|_\text{PF}$.
In particular $|f^p_\#(\nu)|_\text{PF} = |\beta|_\text{PF} = |\nu|_\text{PF}$.
However we build $\nu$ in such a way that $|f^p_\#(\nu)|_\text{PF} = \lambda^p |\nu|_\text{PF}$.
This contradicts our original assumption.
Therefore for every $p \in \N$, $\phi^p(g)\not\equiv g$ in $\burn rn$.
In particular $\phi$ (\resp $\Phi$) induces an automorphism (\resp outer automorphism) of $\burn rn$ of infinite order.
\end{proof}

\section{Comments and questions}

\subsection{About other possible strategies of proof}

\paragraph{}
In the introduction, we recalled the argument given by E.A.~Cherepanov.
It is easy to elaborate a generalization to a wider class of automorphisms
which does not require the criterium stated in \autoref{res: criterium}.

\paragraph{}
An outer automorphism $\Phi\in\out{\free{r}}$ is \emph{irreducible with irreducible powers} (or simply \emph{iwip}) if there is no (conjugacy class of a) proper free factor of $\free{r}$ which is invariant by some positive power of $\Phi$.
An \emph{iwip} outer automorphism can be represented by an (absolute) train-track map $f : G \rightarrow G$ with a primitive transition matrix \cite{BesHan92}. 
Roughly speaking it implies that there is no ``yellow strata'' which were the ones responsible for having large powers in our words.
As a particular case of \autoref{res: no power in the red}, there exists a loop $\gamma$ in $G$ and an integer $n_2$ with the following property.
For every $p \in \N$, the word labeling the loop $f_\#^p(\gamma)$ does not contain an $n_2$-th power (as a complete word, not just its red part).
Consequently \autoref{res: criterium Novikov Adian} is sufficient to conclude.
Note also that, in this context,  \autoref{res: no power in the red} can be proved in a much easier way.
Using either the action of $\free r$ on the stable tree associated to $\Phi$ (see \cite[Theorem 2.1]{GabJaeLev98}) or the fact that the attractive laminations of $\Phi$ cannot by carried by a subgroup of rank $1$ (see \cite[Proposition 2.4]{BesFeiHan97}).

\paragraph{}
However as we explained in the introduction there exists automorphisms for which one cannot use the same strategy.
Consider for instance the automorphism $\psi$ of $\free 4 = \mathbf F (a,b,c,d)$ defined in the introduction by $\psi(a)=a$, $\psi(b)=ba$, $\psi(c)=cbcd$ and $\psi(d)=c$.
One can view $\psi$ as a relative train-track map on the rose: there is only one exponential stratum (the ``red stratum'' which corresponds to the free factor $\langle c,d\rangle$) and the restriction of $\psi$ to $\langle a,b\rangle$ has polynomial growth (and $\langle a,b\rangle$ gives rise to a ``yellow stratum''). 
We saw that $a^{p-1}$ occurs as subword of $\psi^p(d)$.
Nevertheless, we still do not need \autoref{res: criterium} to conclude here that the automorphism $\psi$ satisfies the statement of  \autoref{theo:main}.
It is sufficient to pass to the quotients of $\free{r}$ and $\burn{r}{n}$ by the normal subgroup generated by $a$ and $b$, and then to argue as previously.

\paragraph{}
Nevertheless, given an arbitrary automorphism, this trick (passing to a well chosen quotient) seems to be less easy to run.
Look at the automorphism $\psi$ of $\free 4 = \mathbf F (a,b,c,d)$ defined as follows.
\begin{displaymath}
	\begin{array}{crcl}
		\psi: & a & \rightarrow & a \\  
		& b & \rightarrow & ba \\ 
		& c & \rightarrow & cd\inv bd  \\ 
		& d & \rightarrow & dc d\inv bd.
	\end{array}
\end{displaymath}
This automorphism grows exponentially.
However if one considers the quotient of $\free 4$ by the normal subgroup generated by $a$ and $b$, it induces the Dehn twist $c \rightarrow c$, $d\rightarrow dc$, which has finite order as an automorphism of $\burn 2n$.

\paragraph{} Let $\phi$ be an automorphism of $\free r$.
The geometry of the suspension $\sdp[\phi] {\free r}\Z$ might provide an alternative proof of \autoref{theo:main}.
In \cite{Cou10a}, the first author solved indeed the case where $\sdp[\phi]{\free r}\Z$ is a hyperbolic group.
Generalizing the Delzant-Gromov approach of the Burnside Problem, he constructed a sequence of groups $\H_j$ with $\dlim \H_j = \burn rn$ such that for every $j$,
\begin{itemize}
	\item $\phi$ induces an automorphism of infinite order of $\H_j$,
	\item $\sdp[\phi]{\H_j}\Z$ is a hyperbolic group obtained from $\sdp[\phi]{\H_{j-1}}\Z$ by small cancellation.
\end{itemize}
It follows from the hyperbolicity that $\phi$ induces an automorphism of infinite order of $\burn rn$.

\paragraph{} If $\phi$ is an arbitrary exponentially growing automorphism then $\sdp[\phi]{\free r}\Z$ is no more hyperbolic.
However F.~Gautero and M.~Lustig proved that $\sdp[\phi]{\free r}\Z$ is hyperbolic relatively to a family of subgroups which consists of conjugacy classes that grow polynomially under iteration by $\phi$ \cite{Gautero:2007vp,Gautero:2007uua}.
Therefore one could use a generalization of the iterated small cancellation theory to relative hyperbolic groups.
We refer the reader to \cite{Coulon:2013tx} for a detailed presentation of the Delzant-Gromov approach to the Burnside problem and \cite{Coulon:2013ty} for a generalization.
See also \cite{Dahmani:2011vu} for a theory of small cancellation in relatively hyperbolic groups.

\subsection{Quotients of $\text{Out}({\free {r}})$}

\paragraph{}
The following remark is due to M.~Sapir.
\autoref{prop:pol} says that for every integer $n \geq 1$, polynomially growing automorphisms of $\free r$ induce automorphisms of finite order of $\burn rn$.
More precisely their orders divide 
\begin{equation*}
p(r,n)  = n^{2(2^{r-1}-1)}.
\end{equation*}
Let us denote by $\mathcal Q_{r,n}$ the quotient of $\out{\free r}$ by the (normal) subgroup generated by
\begin{equation*}
	\left\{\Phi^{p(r,n)}| \Phi \in \out{\free r} \text{ polynomially growing} \right\}.
\end{equation*}
In particular the $p(r,n)$-th power of the Nielsen transformations which generate $\out{\free r}$ are trivial in $\mathcal Q_{r,n}$.
It follows from \autoref{prop:pol} that the map $\out{\free r} \rightarrow \out{\burn rn}$ induces a natural map $\mathcal Q_{r,n} \rightarrow \out{\burn rn}$.
Therefore we have the following results

\begin{theo}
	Let $r \geq 3$.
	There exists $n_0$ such that for all odd integers $n \geq n_0$, the group $\mathcal Q_{r,n}$ contains copies of $\free 2$ and $\Z^{\lfloor r/2 \rfloor}$.
\end{theo}

\begin{proof}
	This is a consequence of \cite{Cou10a} Theorems~1.8 and 1.10.
\end{proof}

\begin{theo}
	Let $\Phi$ be an outer automorphism of $\free r$. 
	The following assertions are equivalent:
	\begin{enumerate}
		\item 
		\label{enu: mth exp growth 1}
		$\Phi$ has exponential growth;
		\item 
		\label{enu: mth one exponent 1}
		there exists $n\in\N$ such that the image of $\Phi$ in $\mathcal Q_{r,n}$ has infinite order;
		\item 
		\label{enu: mth all odd exponents 1	}
		there exist $\kappa, n_0\in\N$ such that for all odd integers $n\geq n_0$, the image of $\Phi$ in $\mathcal Q_{r,\kappa n}$ has infinite order.
	\end{enumerate}

\end{theo}

\begin{proof}
This is a consequence of our main theorem.
\end{proof}

\subsection{Exponentially growing automorphisms of the free group can have finite order in a free Burnside group}

\paragraph{}
The constant $n_0$ in \autoref{theo:main} does depend on the outer automorphism $\Phi\in\out{\free{r}}$.
Indeed, we give in this section explicit examples of automorphisms in the kernel of the natural map $\aut{\free{r}}\rightarrow \aut{\burn{r}{n}}$ which have exponential growth. In particular, there are \emph{iwip} automorphisms in this kernel.

\subsubsection{A first family of examples}

\paragraph{}
An outer automorphism $\Phi\in\out{\free{r}}$ induces, by abelianisation, an automorphism of $\Z^r$. 
This defines a homomorphism: $\out{\free{r}}\rightarrow GL(r,\Z)$, $\Phi\mapsto M_\Phi$.
Nielsen proved that for $r=2$, this morphism is an isomorphism \cite{Nielsen:1917du}.
Moreover, $\Phi$ has exponential growth if and only if the absolute value of the trace of $M^2_\Phi\in GL(2,\Z)$ is larger than 2.
\exs
	Let $\{a,b\}$ be a basis of the free group $\free{2}$.
	For $n\in\N^*$,  we define $\phi_n\in\aut{\free{2}}$ by $\phi_n(a)=a(ba^n)^n$, $\phi_n(b)=ba^n$.
	We denote by $\Phi_n$ the corresponding outer class in $\out{\free{2}}$.
	The outer class $\Phi_n$ has exponential growth since the trace of $M^2_{\Phi_n}$ equals $n^4+4n^2+2$. 
	However, the outer automorphism of $\burn{2}{n}$ induced by $\Phi_n$ is the identity.

	\paragraph{}
	For $r>2$, we consider a splitting of $\free{r}$ as a free product $\free{r}=\free{2}*\free{r-2}$.
	For $n\in\N^*$,  we consider the automorphism $\psi_n=\phi_n*Id$ which is equal 	to $\phi_n$ (defined in the previous paragraph) when restricted to the first factor of the splitting and to the identity when restricted to the second factor.
	Again, the outer class $\Psi_n$ of $\psi_n$ has exponential growth (since $\Phi_n$ has), but 	the  the outer automorphism of $\burn{r}{n}$ induced by $\Psi_n$ is the identity.

\paragraph{}
These examples show that the constant $n_0$ in \autoref{theo:main} is not uniform: it does depend on the outer class $\Phi\in\out{\free{r}}$.
The automorphisms $\Phi_n$ are iwip automorphisms.
But this is not the case of the automorphisms $\Psi_n$. 
We fix this point in the next subsection.

\subsubsection{Iwip automorphisms of $\free{r}$ trivial in $\text{Out}(\burn{r}{n})$}

\paragraph{}To produce \emph{iwip} automorphisms in the kernel of the canonical map $\out{\free{r}} \rightarrow \out{\burn{r}{n}}$, one can follow the idea of W.~Thurston to generically produce pseudo-Anosov homeomorphisms of a surface by composing well chosen Dehn twist homeomorphisms \cite{Thu88}.

\paragraph{}In the context of automorphisms of free groups, there is a notion of Dehn twist (outer) automorphism -- see for instance \cite{CohLus95} -- which generalizes the notion of a Dehn twist homeomorphism of a surface:
\autoref{exa: dehn-twist} provides such a Dehn twist automorphism.
In \cite{Clay:2010ie},  M.~Clay and A.~Pettet explain how to generate \emph{iwip} automorphisms of $\free{r}$ by composing two Dehn twist automorphisms associated to a filling pair of cyclic splittings of $\free{r}$.
We will not explicit these definitions here.
Four our purpose we only need to know that:
\begin{itemize}
  \item Dehn twist automorphisms have polynomial growth (in fact linear growth),
and 
  \item there exist Dehn twist automorphisms $\Delta_1,\Delta_2 \in \out{\free r}$ satisfying the hypothesis of the following theorem.
\end{itemize}

\begin{theo}[Clay-Pettet {\cite[Theorem 5.3]{Clay:2010ie}}]
\label{thm:clay-pettet}
	Let $\Delta_1,\Delta_2 \in \out{\free r}$ be the Dehn twist outer automorphisms for a filling pair of cyclic splittings of $\free r$.
	There exists $N\in\N$ such that, for every $p,q>N$:
	\begin{itemize}
	\item the subgroup of $\out{\free{r}}$ generated by  	$\Delta_1^p$ and $\Delta_2^q$ is a free group of rank $2$,
	\item if $\Phi\in<\Delta_1^p,\Delta_2^q>$ is not conjugate to a power of either $\Delta_1^p$ or $\Delta_2^q$, then $\Phi$ is an iwip outer automorphism.
	\end{itemize}
\end{theo}

We fix an exponent $n\in\N$.
We consider two such Dehn twist outer automorphisms $\Delta_1$ and $\Delta_2$ and the integer $N\in\N$ given by \autoref{thm:clay-pettet}.
Since $\Delta_1$ and $\Delta_2$ have polynomial growth, they induce automorphism of finite order of $\burn{r}{n}$.
In particular, there exists $p>N$ such that $\Phi=\Delta_1^p\Delta_2^p$ is in the kernel of the map $\out{\free{r}}\rightarrow \out{\burn{r}{n}}$.
However \autoref{thm:clay-pettet} ensures that $\Phi$ is an \emph{iwip} outer automorphism of $\free r$.

\subsection{Growth rates in $\text{Out}(\free{r})$ and $\text{Out}(\burn{r}{n})$}

\paragraph{}
Let $\Phi$ be an exponentially growing automorphism of $\free r$.
Our study in \autoref{sec:proof main} seems to indicate that for odd exponents $n$ large enough some structure of $\Phi$ is preserved in $\burn rn$.
Therefore we wonder how much information could be carried through the map $\out{\free r}\rightarrow \out{\burn rn}$.
In particular what can we say about the growth rate of $\Phi$?

\paragraph{}
Let $\G$ be a group generated by a finite set $S$.
We endow $\G$ with the word-metric with respect to $S$.
The length of the conjugacy class of $g \in \G$, denoted by $\|g \|$, is the length of the shortest element conjugated to $g$.
An outer automorphism $\Phi$ of $\G$ naturally acts on the set of conjugacy classes of $\G$.
Consequently, as in the free group, one can define the \emph{(exponential) growth rate} of $\Phi$ by
\begin{displaymath}
	EGR(\Phi) = \sup_{g \in \G} \limsup_{p \rightarrow + \infty} \sqrt[p]{\left\| \Phi^p(g)\right\|}.
\end{displaymath}
Since the word-metrics for two distinct finite generating sets of $\G$ are bi-lipschitz equivalent, this rate does not depend on $S$.
The automorphism $\Phi$ is said to have \emph{exponential growth} if $EGR(\Phi)>1$.

\paragraph{} For the extend of our knowledge it is not known if there exist outer automorphisms of Burnside groups with exponential growth.
We would like to ask the following questions. 
\begin{itemize}
	\item Are there automorphisms of $\burn rn$ with exponential growth?
	\item Let $\Phi \in \out{\free r}$ with exponential growth. Is there an integer $n_0$ such that for all (odd) exponents $n \geq n_0$, the automorphism $\hat \Phi_n$ of $\burn rn$ induced by $\Phi$ has exponential growth? Such that $EGR( \hat \Phi_n) = EGR(\Phi)$?
	\item Are there automorphisms of $\burn rn$ of infinite order which do not have exponential growth?
\end{itemize}

\paragraph{}
On the other hand, it could be very interesting to understand to what extend the structure of the attracting laminations associated to an outer automorphism of $\free r$ is preserved in $\burn rn$.
Recall that theses laminations are the fundamental tool used by M.~Bestvina, M.~Feighn and M.~Handel to prove that $\out{\free r}$ satisfies the Tits alternative \cite{BesFeiHan00,BesFeiHan05}.

\makebiblio

\noindent
\emph{R\'emi Coulon} \\
Department of Mathematics, Vanderbilt University\\
Stevenson Center 1326, Nashville TN 37240, USA\\
\texttt{remi.coulon@vanderbilt.edu} \\
\texttt{http://www.math.vanderbilt.edu/$\sim$coulonrb/}

\vskip 5mm

\noindent
\emph{Arnaud Hilion} \\
Aix Marseille Université, CNRS, Centrale Marseille, LATP, UMR 7353 \\
13453 Marseille, France \\
\texttt{arnaud.hilion@univ-amu.fr} \\
\texttt{http://junon.u-3mrs.fr/hilion/}


\begin{thebibliography}{10}

\bibitem{Adi79}
S.~I. Adian.
\newblock {\em {The Burnside problem and identities in groups}}, volume~95 of
  {\em Ergebnisse der Mathematik und ihrer Grenzgebiete [Results in Mathematics
  and Related Areas]}.
\newblock Springer-Verlag, Berlin, 1979.

\bibitem{Arnoux:2006wt}
P.~Arnoux, V.~Berth{\'e}, A.~Hilion, and A.~Siegel.
\newblock {Fractal representation of the attractive lamination of an
  automorphism of the free group}.
\newblock {\em Universit{\'e} de Grenoble. Annales de l'Institut Fourier},
  56(7):2161--2212, 2006.

\bibitem{BesFeiHan97}
M.~Bestvina, M.~Feighn, and M.~Handel.
\newblock {Laminations, trees, and irreducible automorphisms of free groups}.
\newblock {\em Geometric and Functional Analysis}, 7(2):215--244, 1997.

\bibitem{BesFeiHan00}
M.~Bestvina, M.~Feighn, and M.~Handel.
\newblock {The Tits alternative for $\rm Out(F\sb n)$. I. Dynamics of
  exponentially-growing automorphisms}.
\newblock {\em Annals of Mathematics. Second Series}, 151(2):517--623, 2000.

\bibitem{BesFeiHan05}
M.~Bestvina, M.~Feighn, and M.~Handel.
\newblock {The Tits alternative for $\rm Out(F\sb n)$. II. A Kolchin type
  theorem}.
\newblock {\em Annals of Mathematics. Second Series}, 161(1):1--59, 2005.

\bibitem{BesHan92}
M.~Bestvina and M.~Handel.
\newblock {Train tracks and automorphisms of free groups}.
\newblock {\em Annals of Mathematics. Second Series}, 135(1):1--51, 1992.

\bibitem{Bur02}
W.~Burnside.
\newblock {On an unsettled question in the theory of discontinuous groups}.
\newblock {\em Quart.J.Math.}, 33:230--238, 1902.

\bibitem{Che05}
E.~A. Cherepanov.
\newblock {Free semigroup in the group of automorphisms of the free Burnside
  group}.
\newblock {\em Communications in Algebra}, 33(2):539--547, 2005.

\bibitem{Clay:2010ie}
M.~Clay and A.~Pettet.
\newblock {Twisting out fully irreducible automorphisms}.
\newblock {\em Geometric and Functional Analysis}, 20(3):657--689, 2010.

\bibitem{CohLus95}
M.~M. Cohen and M.~Lustig.
\newblock {Very small group actions on $\bf R$-trees and Dehn twist
  automorphisms}.
\newblock {\em Topology. An International Journal of Mathematics},
  34(3):575--617, 1995.

\bibitem{Cou10a}
R.~Coulon.
\newblock {Outer automorphisms of free Burnside groups}.
\newblock {\em Commentarii Mathematici Helvetici, to appear arXivorg},
  (1008.4495v1), 2010.

\bibitem{Coulon:2012tj}
R.~Coulon.
\newblock {Detecting trivial elements of Burnside groups}.
\newblock {\em arXiv.org}, (1211.4267v2), Nov. 2012.

\bibitem{Coulon:2013ty}
R.~Coulon.
\newblock {Partial periodic quotient of groups acting on a hyperbolic space}.
\newblock {\em in preparation}, 2013.

\bibitem{Coulon:2013tx}
R.~Coulon.
\newblock {Small cancellation theory and Burnside problem}.
\newblock {\em arXiv.org}, (1302.6933v2), Feb. 2013.

\bibitem{CulMor87}
M.~Culler and J.~W. Morgan.
\newblock {Group actions on $\bf R$-trees}.
\newblock {\em Proceedings of the London Mathematical Society. Third Series},
  55(3):571--604, 1987.

\bibitem{Dahmani:2011vu}
F.~Dahmani, V.~Guirardel, and D.~V. Osin.
\newblock {Hyperbolically embedded subgroups and rotating families in groups
  acting on hyperbolic spaces}.
\newblock {\em arXiv.org}, (1111.7048v2), Nov. 2011.

\bibitem{DelGro08}
T.~Delzant and M.~Gromov.
\newblock {Courbure m{\'e}soscopique et th{\'e}orie de la toute petite
  simplification}.
\newblock {\em Journal of Topology}, 1(4):804--836, 2008.

\bibitem{GabJaeLev98}
D.~Gaboriau, A.~Jaeger, G.~Levitt, and M.~Lustig.
\newblock {An index for counting fixed points of automorphisms of free groups}.
\newblock {\em Duke Mathematical Journal}, 93(3):425--452, 1998.

\bibitem{Gautero:2007vp}
F.~Gautero.
\newblock {Combinatorial mapping-torus, branched surfaces and free group
  automorphisms}.
\newblock {\em Annali della Scuola Normale Superiore di Pisa. Classe di
  Scienze. Serie V}, 6(3):405--440, 2007.

\bibitem{Gautero:2007uua}
F.~Gautero and M.~Lustig.
\newblock {The mapping-torus of a free group automorphism is hyperbolic
  relative to the canonical subgroups of polynomial growth}.
\newblock {\em arXiv.org}, (0707.0822v2), July 2007.

\bibitem{Hal57}
M.~Hall, Jr.
\newblock {Solution of the Burnside problem of exponent $6$}.
\newblock {\em Proceedings of the National Academy of Sciences of the United
  States of America}, 43:751--753, 1957.

\bibitem{Iva94}
S.~V. Ivanov.
\newblock {The free Burnside groups of sufficiently large exponents}.
\newblock {\em International Journal of Algebra and Computation},
  4(1-2):ii+308, 1994.

\bibitem{Karhumaki:1983fy}
J.~Karhum{\"a}ki.
\newblock {On cube-free $\omega $-words generated by binary morphisms}.
\newblock {\em Discrete Applied Mathematics}, 5(3):279--297, 1983.

\bibitem{LevWae33}
F.~Levi and B.~van~der Waerden.
\newblock {\"Uber eine besondere Klasse von Gruppen}.
\newblock {\em Abhandlungen aus dem Mathematischen Seminar der Universit\"at
  Hamburg}, 9(1):154--158, Dec. 1933.

\bibitem{Lev09}
G.~Levitt.
\newblock {Counting growth types of automorphisms of free groups}.
\newblock {\em Geometric and Functional Analysis}, 19(4):1119--1146, 2009.

\bibitem{Lys96}
I.~G. Lysenok.
\newblock {Infinite Burnside groups of even period}.
\newblock {\em Izvestiya Akademii Nauk SSSR. Seriya Matematicheskaya},
  60(3):3--224, 1996.

\bibitem{Mos92}
B.~Moss{\'e}.
\newblock {Puissances de mots et reconnaissabilit{\'e} des points fixes d'une
  substitution}.
\newblock {\em Theoretical Computer Science}, 99(2):327--334, 1992.

\bibitem{Nielsen:1917du}
J.~Nielsen.
\newblock {Die Isomorphismen der allgemeinen, unendlichen Gruppe mit zwei
  Erzeugenden}.
\newblock {\em Mathematische Annalen}, 78(1):385--397, 1917.

\bibitem{NovAdj68c}
P.~S. Novikov and S.~I. Adian.
\newblock {Infinite periodic groups.}
\newblock {\em Izvestiya Akademii Nauk SSSR. Seriya Matematicheskaya}, 32,
  1968.

\bibitem{Olc82}
A.~Y. Ol'shanski{\u \i}.
\newblock {The Novikov-Adyan theorem}.
\newblock {\em Matematicheski\u\i\ Sbornik}, 118(160)(2):203--235, 287, 1982.

\bibitem{San40}
I.~N. Sanov.
\newblock {Solution of Burnside's problem for exponent 4}.
\newblock {\em Leningrad State Univ. Annals [Uchenye Zapiski] Math. Ser.},
  10:166--170, 1940.

\bibitem{Seneta:2006wd}
E.~Seneta.
\newblock {\em {Non-negative matrices and Markov chains}}.
\newblock Springer Series in Statistics. Springer, New York, 2006.

\bibitem{Thu88}
W.~P. Thurston.
\newblock {On the geometry and dynamics of diffeomorphisms of surfaces}.
\newblock {\em American Mathematical Society. Bulletin. New Series},
  19(2):417--431, 1988.

\end{thebibliography}
\end{document}